%% file: mainLIPIcs.tex
\pdfoutput=1
\documentclass[a4paper,UKenglish,cleveref, autoref,thm-restate]{lipics-v2021}

\title{Fragility and Robustness in Mean-payoff Adversarial Stackelberg Games} 

\titlerunning{Fragility and Robustness in Mean-payoff Adversarial Stackelberg Games} 

\author{Mrudula Balachander}{Universit\'{e} libre de Bruxelles, Belgium}{}{}{}
\author{Shibashis Guha}{Tata Institute of Fundamental Research, India}{}{}{}
\author{Jean-François Raskin}{Universit\'{e} libre de Bruxelles, Belgium}{}{}{}

\authorrunning{Balachander, Guha, Raskin} 

\Copyright{Mrudula Balachander, Shibashis Guha, Jean-François Raskin} 



\category{} 

\relatedversion{} 

\funding{This work is partially supported by the PDR project Subgame perfection in graph games (F.R.S.-FNRS), the ARC project Non-Zero Sum Game Graphs: Applications to Reactive Synthesis and Beyond (Fédération Wallonie-Bruxelles), the EOS project Verifying Learning Artificial Intelligence Systems (F.R.S.-FNRS and FWO), and the COST Action 16228 GAMENET (European Cooperation in Science and Technology).}

\nolinenumbers 

\hideLIPIcs  

\usepackage[utf8]{inputenc}
\usepackage{boxedminipage}
\usepackage{graphicx}
\usepackage{pstricks}
\usepackage[]{todonotes}

\usepackage{multirow}
\usepackage{amsmath}
\usepackage{amssymb}

\usepackage{verbatim}
\usepackage{paralist}
\usepackage{xcolor}
\usepackage{tkz-euclide}
\usepackage{changepage}
\usepackage{setspace}
\usetikzlibrary{arrows,automata,positioning}

\newcounter{casenum}
\newenvironment{caseof}{\setcounter{casenum}{1}}{\vskip.5\baselineskip}
\newcommand{\case}[2]{\vskip.5\baselineskip\par\noindent {\bfseries Case \arabic{casenum}:} #1\\#2\addtocounter{casenum}{1}}

\newcommand{\track}[1]{{\textcolor{red}{#1}}}
\newcommand{\btrack}[1]{{\textcolor{blue}{#1}}}

\newcommand{\CS}[1][S]{\mathbb{C}(#1)}
\newcommand{\CShat}{\mathbb{C}(\widehat{S})}
\newcommand{\CH}[1]{{\sf CH}(#1)}
\newcommand{\Fmin}[1]{{\sf F_{\min}}{(#1)}}
\newcommand{\poly}{{\sf poly}}
\newcommand{\first}{{\sf first}}
\newcommand{\last}{{\sf last}}
\newcommand{\val}{{\sf Val}}
\newcommand{\asv}{\asvgen^{\epsilon}}
\newcommand{\asvgen}{\mathsf{ASV}}
\newcommand{\mpinf}{\underline{\mpgen}}
\newcommand{\mpgen}{\mathsf{MP}}
\newcommand{\plays}{\mathsf{Plays}_{\mathcal{A}}}
\newcommand{\fm}{\mathsf{FM}}
\newcommand{\ml}{\mathsf{ML}}
\newcommand{\br}{\brgen^{\epsilon}}
\newcommand{\brgen}{\mathsf{BR}_1}
\newcommand{\outv}{\mathsf{Out}_v}
\newcommand{\outgen}{\mathsf{Out}}
\newcommand{\hist}{\mathsf{Hist}}

\newcommand{\zug}[1]{\langle #1  \rangle}
\newcommand{\stam}[1]{}

\newcommand{\mychapter}{{section}}

\begin{document}

\maketitle

\begin{abstract}
\input{abstract}
\end{abstract}

\section{Introduction} 
  \label{sec:intro}
  \input{intro2}

\section{Preliminaries}
  \label{sec:prelim}
  \input{Preliminaries}

\section{Fragility and robustness in games}
  \label{sec:robustness}
  \input{Robustness}

\section{Threshold problem for the $\asv$}
  \label{sec:ThresholdProblem}
  \input{DecisionAlgorithmsASV}

  \input{FiniteMemForASV}
\section{Computation of the $\asv$ and the largest $\epsilon$ possible}
  \label{sec:ComputeASV}
  \input{ComputeASV}

\section{Additional Properties of $\asv$}
  \label{sec:additional_prop}
  \input{intro_additional_prop}
  \input{Achievability_v2}
  \input{ASVEpsilonExamples}

\bibliography{refs}

\newpage
\appendix
\input{Appendix/Threshold_Problem/index}

\end{document}

%% file: abstract.tex
Two-player mean-payoff Stackelberg games are nonzero-sum infinite duration games played on a bi-weighted graph by Leader (Player 0) and Follower (Player 1). Such games are played sequentially: first, Leader announces her strategy, second, Follower chooses his best-response. If we cannot impose which best-response is chosen by Follower, we say that Follower, though strategic, is adversarial towards Leader. The maximal value that Leader can get in this nonzero-sum game is called the adversarial Stackelberg value (ASV) of the game.

We study the robustness of strategies for Leader in these games against two types of deviations: (i) Modeling imprecision - the weights on the edges of the game arena may not be exactly correct, they may be delta-away from the right one. (ii) Sub-optimal response - Follower may play epsilon-optimal best-responses instead of perfect best-responses. First, we show that if the game is zero-sum then robustness is guaranteed while in the nonzero-sum case, optimal strategies for ASV are fragile. Second, we provide a solution concept to obtain strategies for Leader that are robust to both modeling imprecision, and as well as to the epsilon-optimal responses of Follower, and study several properties and algorithmic problems related to this solution concept.

%% file: intro2.tex
\noindent Stackelberg games \cite{S34} were first introduced to model strategic interactions among rational agents in markets that consist of Leader and Follower(s).
Leader in the market makes her strategy public and Follower(s) respond by playing an optimal response to this strategy.
Here, we consider Stackelberg games as a framework for the synthesis of reactive programs \cite{PR89,BCHPRRS16}.
These programs maintain a continuous interaction with the environment in which they operate; they are \emph{deterministic functions} that given a history of interactions so far choose an action.
Our work is a contribution to \emph{rational synthesis} \cite{FKL10,KPV16}, a nonzero-sum game setting where both the program and the environment are considered as rational agents that have their own goals.
While Boolean $\omega$-regular payoff functions have been studied in \cite{FKL10,KPV16}, here we study the quantitative long-run average (mean-payoff) function. 

We illustrate our setting with the example of a game graph as shown in \cref{fig:robust-game-example}. 
The set $V$ of vertices is partitioned into $V_0$ (represented by circles) and $V_1$ (represented by squares) that are owned by Leader (also called Player~0) and Follower (also called Player~1) respectively.
In the tuple on the edges, the first element is the payoff of Leader, while the second one is the payoff of Follower (weights are omitted if they are both equal to 0).
Each player's objective is to maximize the long run average of the payoffs that she receives (a.k.a. mean-payoff). 
In the adversarial Stackelberg setting, Player~0 (Leader) first announces how she will play then Player~1 (Follower) chooses one of his best-responses to this strategy. Here, there are two choices for Player~0: $L$ or $R$. As Player~1 is assumed to be rational, Player~0 deduces that she must play $L$. Indeed, the best response of Player~1 is then to play $LL$ and the reward she obtains is $10$. This is better than playing $R$, for which the best-response of Player~1 is $RL$, and the reward is $8$ instead of $10$. Note that if there are several possible best responses for Player~1, then we consider the worst-case: Player~0 has no control on the choice of best-responses by Player~1.

\subparagraph{{\bf Quantitative models and robustness}}
The study of adversarial Stackelberg games with mean-payoff objectives has been started in~\cite{FGR20} with the concept of {\em adversarial Stackelberg value} ($\asvgen$ for short). $\asvgen$ is the best value that Leader can obtain by fixing her strategy and facing any rational response by Follower. As this setting is quantitative, it  naturally triggers questions about {\em robustness} that were left open in the above paper.

Robustness is a {\em highly desirable} property of quantitative models: small changes in the quantities appearing in a model $M$ (e.g. rewards, probabilities, etc.) should have small impacts on the predictions made from $M$, see e.g.~\cite{BCGHHJKK14}.
Robustness is thus crucial because it accounts for modelling imprecision that are inherent in quantitative modelling and those imprecision may have important consequences. 
For instance, a reactive program synthesized from a model $M$ should provide acceptable performances if it is executed in a real environment that differ slightly w.r.t. the quantities modeled in $M$.

Some classes of models are robust. 
For instance, consider two-player {\em zero-sum} mean-payoff games where players have fully antagonistic objectives.
The \emph{value} of a two-player zero-sum mean-payoff $\mathcal{G}$ is the maximum mean-payoff that Player~0 can ensure against all strategies of Player~1.
A strategy $\sigma_0$ that enforces the optimal value $c$ in $\mathcal{G}$ is {\em robust} in the following sense. 
Let $\mathcal{G}^{\pm \delta}$ be the set of games obtained by increasing or decreasing the weights on the edges of $\mathcal{G}$ by at most $\delta$.
Then for all $\delta > 0$, and for all $\mathcal{H} \in \mathcal{G}^{\pm \delta}$, the strategy $\sigma_0$ ensures in $\mathcal{H}$ a mean-payoff of at least $c-\delta$ for Player~0 against any strategy of Player~1 (Theorem~\ref{robustZS}). So slight changes in the quantities appearing in the model have only a small impact on the worst-case value enforced by the strategy. 

The situation is more complex and less satisfactory in {\em nonzero-sum} games. Strategies that enforce the $\asvgen$ proposed in~\cite{FGR20} may be {\em fragile}: slight differences in the weights of the game, or in the optimality of the response by Player~1, may lead to large differences in the value obtained by the strategy.
We illustrate these difficulties on our running example. The strategy of Player~0 that chooses $L$ in $v_0$ ensures her a payoff of $10$ which is the $\asvgen$.
Indeed, the unique best-response of Player~1 against $L$ is to play $LL$ from $v_1$. However, if the weights in $\mathcal{G}$ are changed by up to $\pm \delta=\pm 0.6$
then there is a game $\mathcal{H} \in \mathcal{G}^{\pm \delta}$ in which the weight on the self-loop over vertex $v_4$ changes to e.g. $9.55$, and the weight on the self-loop over $v_3$ changes to e.g. $9.45$, and the action $LR$ becomes better for Player~1.
So the value of $L$ in $\mathcal{H}$ against a rational adversary is now $0$ instead of $10$. Thus a slight change in the rewards for Player~1 (due to e.g. modelling imprecision) may have a dramatic effect on the value of the optimal strategy $L$ computed on the model $\mathcal{G}$ when evaluated in $\mathcal{H}$.

\subparagraph{{\bf Contributions}}

As a remedy to this situation, we provide an alternative notion of value that is better-suited to synthesize strategies that are robust against perturbations. We consider two types of perturbations. First, the strategies computed for this value are robust against {\em modeling imprecision}: if a strategy has been synthesized from a weighted game graph with weights that are possibly slightly wrong, the value that this strategy delivers is guaranteed to be close to what the model predicts. Second,  strategies computed for this value are robust against {\em sub-optimal responses}: small deviations from the best-response by the adversary have only limited effect on the value guaranteed by the strategy.

Our solution relies on relaxing the notion of best-responses of Player~1 \emph{in the original model $\mathcal{G}$}. More precisely, we define the $\epsilon$-adversarial Stackelberg  value ($\asv$, for short) as the value that Leader can enforce against all $\epsilon$-best responses of Follower. Obviously, this directly accounts for the second type of perturbations. But we show that, additionally, this accounts for the first type of perturbations: if a strategy $\sigma_0$ enforces an $\asv$ equal to $c$ then for all games $\mathcal{H} \in G^{\pm \frac{\epsilon}{2}}$, we have that $\sigma_0$ enforce a value larger than $c-\epsilon$ in $\mathcal{H}$ (Theorem~\ref{thm:combined-robustness} and Theorem~\ref{robustNZS}). 

We illustrate this by considering again the example of Figure~\ref{fig:robust-game-example}. 
Here, if we consider that the adversary can play $2\delta=1.2$-best responses
instead of best responses only, then the optimal strategy of Player~0 is now $R$ and it has a $\asv$ equal to $8$. This value is guaranteed to be robust for all games $\mathcal{H} \in \mathcal{G}^{\pm \delta}$ as $R$ is guaranteed to enforce a payoff that is larger than $8-\delta$ in all games in $\mathcal{G}^{\pm \delta}$.  
Stated otherwise, we use the notion of $\asv$ in the original game to find a strategy for Player~0 that she uses in the perturbed model while playing against a rational adversary.
Thus we show that in the event of modelling imprecision resulting in a perturbed model, the solution concept to be used is $\asv$ instead of $\asvgen$ since the former provides strategies that are robust to such perturbations.

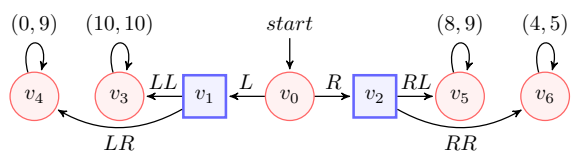
\begin{figure}[t]
\begin{minipage}{1\textwidth}
\centering
\scalebox{.8}{
    \begin{tikzpicture}[->,>=stealth',shorten >=1pt,auto,node distance=1.4cm,
                        semithick, squarednode/.style={rectangle, draw=blue!60, fill=blue!5, very thick, minimum size=7mm}, scale=.8]
    
      \tikzstyle{every state}=[fill=white,draw=black,text=black,minimum size=8mm]
      \node[state, draw=red!60, fill=red!5]                              (A)                {$v_0$};
      \node[squarednode]                                                 (A_1) [right of=A] {$v_2$};
      \node[state, draw=red!60, fill=red!5]                              (B_1) [right of=A_1]  {$v_5$};
      \node[state, draw=red!60, fill=red!5]                              (C_1) [right of=B_1] {$v_6$};
      \node[squarednode]                                                 (A_2) [left of=A] {$v_1$};
      \node[state, draw=red!60, fill=red!5]                              (B_2) [left of=A_2]  {$v_3$};
      \node[state, draw=red!60, fill=red!5]                              (C_2) [left of=B_2] {$v_4$};
      \node[draw=none, fill=none, minimum size=0cm, node distance = 1.2cm] (D) [above of=A] {$start$};
      \path (A)   edge [, above ]          node {$R$}       (A_1)
                  edge [, above ]          node {$L$}       (A_2)
            (A_1) edge [, above ]          node {$RL$}      (B_1)
                  edge [bend right, below] node {$RR$}      (C_1)
            (A_2) edge [, above ]          node {$LL$}      (B_2)
                  edge [bend left]         node {$LR$}      (C_2)
            (B_1) edge [loop above]        node {$(8, 9)$}  (B_1)
            (C_1) edge [loop above]        node {$(4, 5)$}  (C_1)
            (B_2) edge [loop above]        node {$(10, 10)$} (B_2)
            (C_2) edge [loop above]        node {$(0, 9)$}  (C_2)
            (D)   edge [left]              node {}          (A);
    \end{tikzpicture}}
    \caption{A game in which the strategy of Leader that maximizes the adversarial Stackelberg is fragile while the strategy of Leader that maximizes the $\epsilon=1$-adversarial Stackelberg value is robust.\label{fig:robust-game-example}}
    
\end{minipage}
\end{figure}

\begin{table}[hbt]
\scalebox{0.85}{
\begin{tabular}{|c|c|c|c|c|}
\hline
 &
  Robustness &
  Threshold Problem &
  Computing ASV &
  Achievability \\ \hline
\begin{tabular}[c]{@{}c@{}}Adversarial \\ best responses \\ of Follower\end{tabular} &
    \begin{tabular}[c]{@{}c@{}}{ \textbf{No}} \\ \\ { \textbf{[Prop \ref{fragile}, Prop \ref{fragile-tolerance}]}} \end{tabular} &
  \begin{tabular}[c]{@{}c@{}}{ {\sf NP} \cite{FGR20}} \\ \\ { \textbf{Finite}} {\textbf{Memory}} \\ { \textbf{Strategy}} { \textbf{[Thm \ref{LemFinMemWitnessASVNonEps}]}}
  \\ \\
  {\textbf{Memoryless}} \\ {\textbf{Strategy}} 
  {\textbf{[Thm \ref{ThmNPHardMemlessStrat}]}}
  \end{tabular} &
  \begin{tabular}[c]{@{}c@{}}{ Theory of Reals \cite{FGR20}} \end{tabular} &
  { No \cite{FGR20}} \\ \hline
\begin{tabular}[c]{@{}c@{}}Adversarial\\ $\epsilon$-best responses \\ of  Follower\end{tabular} &
    \begin{tabular}[c]{@{}c@{}}{ \textbf{Yes}} \\ \\ { \textbf{[Thm \ref{robustNZS}, Thm \ref{thm:combined-robustness}]}} \end{tabular} &
  { \textbf{\begin{tabular}[c]{@{}c@{}}{\sf \bf NP {[}Thm \ref{ThmNpForASV}{]}}\\ \\ Finite Memory \\ Strategy {[}Thm \ref{ThmNpForASV}{]}
  \\ \\
  {\textbf{Memoryless}} \\ {\textbf{Strategy}} 
  {\textbf{[Thm \ref{ThmNPHardMemlessStrat}]}} \end{tabular}}} &
  { \textbf{\begin{tabular}[c]{@{}c@{}}Theory of Reals\\  {[Thm \ref{ThmComputeASV}]} \\ \\ Solving LP \\ in {\sf EXPTime} \\ {[Thm \ref{ThmComputeASVExpTime}]} \end{tabular}}} &
  { \textbf{\begin{tabular}[c]{@{}c@{}}Yes {[Thm \ref{ThmAchiev}]} \\ \\ (Requires \\ Infinite  Memory\\  {[}Thm \ref{ThmExNeedInfMem}{]}) \end{tabular}}} \\ \hline
\end{tabular}}
    \caption{Summary of our results}
\vspace{-.3cm}
\label{tab:results}
\end{table}

In addition to proving the fragility of the original concept introduced in~\cite{FGR20} (\cref{fragile}) and the introduction of the new notion of value $\asv$ that is robust against modelling imprecision (\cref{robustNZS}), we provide algorithms to handle $\asv$.
First, we show how to decide the threshold problem for $\asv$ in nondeterministic polynomial time and that finite memory strategies suffice (\cref{ThmNpForASV}). Second, we provide an algorithm to compute $\asv$ when $\epsilon$ is fixed (\cref{ThmComputeASV}).
Third, we provide an algorithm that given a threshold value $c$, computes the largest possible $\epsilon$ such that $\asv > c$ (\cref{thm:ComputeEpsilon}). These three results form the core technical contributions of this paper and they are presented in \cref{sec:ThresholdProblem} and \cref{sec:ComputeASV}.
Additionally, in \cref{sec:additional_prop}, we show that $\asv$ is always achievable (\cref{ThmAchiev}), which is in contrast to the case in \cite{FGR20} where Follower only plays best-responses.
Finally, we provide results that concern the memory needed for players to play optimally, and complexity results for subcases (for example when Players are assumed to play memoryless).
Our contributions have been summarized in \cref{tab:results}, where the results obtained in this work are in bold.

\subparagraph{Related Works}
Stackelberg games on graphs have been first considered in \cite{FKL10}, where the authors study rational synthesis for $\omega$-regular objectives with co-operative Follower(s).
In \cite{FGR20}, Stackelberg mean-payoff games in adversarial setting, and  Stackelberg discounted sum games in both adversarial and co-operative setting have been considered. 
However, as pointed out earlier, the model of \cite{FGR20} is not robust to perturbations.
In \cite{GS14}, mean-payoff Stackelberg games in the co-operative setting have been studied. 
In \cite{GS15}, the authors study the effects of limited memory on both Nash and Stackelberg (or leader) strategies in multi-player discounted sum games.
Incentive equilibrium over bi-matrix games and over mean-payoff games in a co-operative setting have been studied in \cite{GS18} and \cite{GSTDP16} respectively.
In \cite{KPV16}, adversarial rational synthesis for $\omega$-regular objectives have been studied.
In \cite{CFGR16}, precise complexity results for various $\omega$-regular objectives have been established for both adversarial and co-operative settings.
In \cite{CHJ06,BMR14}, secure Nash equilibrium has been studied, where each player first maximises her own payoff, and then minimises the payoff of the other player;
Player~0 and Player~1 are symmetric there unlike in Stackelberg games. For discounted sum objectives, in \cite{FGR20}, the gap problem has been studied.
Given rationals $c$ and $\delta$, a solution to the gap problem can decide if $\asvgen > c+\delta$ or $\asvgen < c-\delta$.
The threshold problem was left open in \cite{FGR20}, and is technically challenging.
We leave the case of analysing robustness for discounted sum objective
for future work.

%% file: Preliminaries.tex
\noindent We denote by $\mathbb{N}$, $\mathbb{N}^+$, $\mathbb{Q}$, and $\mathbb{R}$ the set of naturals, the set of naturals excluding $0$, the set of rationals, and the set of reals respectively.

\subparagraph{Arenas} An (bi-weighted) arena $\mathcal{A}=(V, E, \langle V_0, V_1 \rangle, w_0, w_1)$ consists of a finite set $V$ of vertices, a set $E \subseteq V \times V$ of edges such that for all $v \in V$ there exists $v' \in V$ and $(v, v') \in E$, a partition $\langle V_0, V_1 \rangle$ of $V$, where $V_0$ (resp. $V_1$) is the set of vertices for Player~0 (resp. Player~1), and two edge weight functions $w_0 : E \rightarrow \mathbb{Z}$, $w_1 : E \rightarrow \mathbb{Z}$. In the sequel, we denote the maximum absolute value of a weight in $\mathcal{A}$ by $W$.
A strongly connected component of a directed graph is a subgraph that is strongly connected. 
In the sequel, unless otherwise mentioned, we denote by $SCC$ a subgraph that is strongly connected, and which may or may not be maximal.

\subparagraph{Plays and histories} A play in $\mathcal{A}$ is an infinite sequence of vertices $ \pi = \pi_0\pi_1 \dots \in V^{\omega}$ such that for all $k \in \mathbb{N}$, we have $(\pi_k, \pi_{k+1}) \in E$.
 A \textit{history} in $\mathcal{A}$ is a (non-empty) prefix of a play in $\mathcal{A}$. 
Given $ \pi = \pi_0\pi_1\dots \in \plays$ and $k \in \mathbb{N}$, the prefix $\pi_0\pi_1\dots\pi_k$ of $\pi$ is denoted by $\pi_{\leqslant k}$.
We denote by $\inf(\pi)$ the set of vertices $v$ that appear infinitely many times along $\pi$, i.e., $\inf(\pi) = \{v \in V \mid \forall i \in \mathbb{N}\cdot \exists j \in \mathbb{N}, j \geqslant i: \pi(j) = v \}$. It is easy to see that $\inf(\pi)$ forms an SCC in the underlying graph of the arena $\mathcal{A}$.
 We denote by $\plays$ and $\hist_{\mathcal{A}}$ the set of plays and the set of histories in $\mathcal{A}$ respectively; the symbol $\mathcal{A}$ is omitted when clear from the context. Given $i \in \{0,1\}$, the set $\hist^i_{\mathcal{A}}$ denotes the set of histories such that their last vertex belongs to $V_i$. We denote the first vertex and the last vertex of a history $h$ by $\first(h)$ and $\last(h)$ respectively. 

\subparagraph{Games} A \textit{mean-payoff game} $\mathcal{G} = (\mathcal{A},\langle \mpinf_0, \mpinf_1\rangle)$ consists of a bi-weighted arena $\mathcal{A}$, payoff functions $\mpinf_0 : \plays \rightarrow \mathbb{R}$ and $\mpinf_1 : \plays \rightarrow \mathbb{R}$ for for Player~0 and Player~1 respectively which are defined as follows. Given a play $\pi \in \plays$ and $i \in \{0,1\}$, the payoff $\mpinf_i(\pi)$ is given by $\mpinf_i(\pi) = \liminf\limits_{k \to \infty} \frac{1}{k}w_i(\pi_{\leqslant k}) $, where the weight $w_i(h)$ of a history $h \in \hist$ is the sum of the weights assigned by $w_i$ to its edges. 
In our definition of the mean-payoff, we have used $\liminf$ as the limit of the successive average may not exist.
We will also need the $\limsup$ case for technical reasons. Here is the formal definition together with its notation: $\overline{\mpgen}_i(\pi) = \limsup\limits_{k \to \infty} \frac{1}{k}w_i(\pi_{\leqslant k})$.
The size of the game $\mathcal{G}$, denoted $|\mathcal{G}|$, is the sum of the number of vertices and edges appearing in the arena $\mathcal{A}$.

\subparagraph{Unfolding of a game} Let $V$ and $E$ be respectively the set of vertices and the set of edges of $\mathcal{G}$.
The \emph{unfolding} of the game $\mathcal{G}$ starting from a vertex $v \in V$ is a tree $T_v(\mathcal{G})$ of infinite depth with its root $v$ such that there is a one-to-one correspondence between the set of plays $\pi$ of $\mathcal{G}$ with $first(\pi)=v$ and the branches of $T_v(\mathcal{G})$.
Every node $v_1$ of $T_v(\mathcal{G})$ belongs to $V$, and there is an edge from $v_1$ to $v_2$ in $T_v(\mathcal{G})$ iff $(v_1, v_2) \in E$.
Every node $p \in V^{+}$ of $T_v(\mathcal{G})$ is a play $p=v_1 \dots v_n$ in $\mathcal{G}$, where $v_1=v$.
There is an edge from $p=v_1 \dots v_n$ to $p'= v_1 \dots v_n v'_n$ iff $(v_n, v'_n) \in E$.

\subparagraph{Strategies and payoffs} A strategy for Player $i \in \{0,1\}$ in the game $\mathcal{G}$
is a function $\sigma: \hist^i_{\mathcal{A}} \rightarrow V$ that maps histories ending in a vertex $v \in V_i$ to a successor of $v$. The set of all strategies of Player $i \in \{0,1\}$ in the game $\mathcal{G}$ is denoted by $\Sigma_i(\mathcal{G})$, or $\Sigma_i$ when $\mathcal{G}$ is clear from the context.
A strategy has memory $\mathsf{M}$ if it can be realized as the output of a state machine 
with $\mathsf{M}$ states. A memoryless strategy is a function that only depends on the last element of the history $h \in \hist$. We denote by $\Sigma^{\mathsf{ML}}_i$ the set of memoryless strategies of Player i, and by $\Sigma^{\mathsf{FM}}_i$ her set of finite memory strategies. A \textit{profile} is a pair of strategies $\overline{\sigma} = (\sigma_0, \sigma_1)$, where $\sigma_0 \in \Sigma_0(\mathcal{G})$ and $\sigma_1 \in \Sigma_1(\mathcal{G})$. As we consider games with perfect information and deterministic transitions, any profile $\overline{\sigma}$ yields, from any history $h$, a unique play or \textit{outcome}, denoted $\outgen_h(\mathcal{G}, \overline{\sigma})$. 
Formally, $\outgen_h(\mathcal{G}, \overline{\sigma})$ is the play $\pi$ such that $\pi_{\leqslant|h|-1} = h$ and $\forall k \geqslant |h|-1$ it holds that $\pi_{k+1}=\sigma_i(\pi_{\leqslant k})$ if $\pi_k \in V_i$. 
We write $h\leqslant\pi$ whenever $h$ is a prefix of $\pi$.
The set of outcomes compatible with a strategy $\sigma \in \Sigma_{i\in\{0,1\}}(\mathcal{G})$ after a history $h$ is $\outgen_h(\mathcal{G}, \sigma) = \{ \pi | \exists \sigma'\in \Sigma_{1-i}(\mathcal{G}) $ such that $\pi = \outgen_h(\mathcal{G}, (\sigma,\sigma'))\}$.
Each outcome $\pi \in \mathcal{G} = (\mathcal{A},\langle \mpinf_0, \mpinf_1\rangle)$ yields a payoff $\mpinf(\pi)=(\mpinf_0(\pi),\mpinf_1(\pi))$.

Usually, we consider instances of games such that the players start playing at a fixed vertex $v_0$. Thus, we call an initialized game a pair $(\mathcal{G}, v_0)$, where $\mathcal{G}$ is a game and $v_0 \in V$ is the initial vertex. 
When 
$v_0$ is clear from context, we use $\mathcal{G}$, $\outgen(\mathcal{G}, \overline{\sigma})$, $\outgen(\mathcal{G},\sigma)$, $\mpinf(\overline{\sigma})$ instead of $\mathcal{G}_{v_0}$, $\outgen_{v_0}(\mathcal{G}, \overline{\sigma})$, $\outgen_{v_0}(\mathcal{G},\sigma)$, $\mpinf_{v_0}(\overline{\sigma})$. 
We sometimes 
omit $\mathcal{G}$ when it is clear from the context.


\subparagraph{Best-responses, $\epsilon$-best-responses} Let $\mathcal{G} = (\mathcal{A},\langle \mpinf_0, \mpinf_1\rangle)$ be a 
two-dimensional mean-payoff game
on the bi-weighted arena $\mathcal{A}$. Given a strategy $\sigma_0$ for Player~0, we define 
\begin{enumerate}
\item 
Player~1's best responses to $\sigma_0$, denoted by $\brgen(\sigma_0)$, as:
\begin{equation*}
    \{\sigma_1 \in \Sigma_1 \mid \forall v \in V . \forall \sigma'_1 \in \Sigma _1: 
    \mpgen_1(\outv(\sigma_0,\sigma_1)) \geqslant \mpgen_1(\outv(\sigma_0,\sigma'_1))\}
\end{equation*}
\item \label{epsion_gt_0}
Player~1's $\epsilon$-best-responses to $\sigma_0$, for $\epsilon > 0$\footnote{Since we will use $\epsilon$ in $\asv$ to add robustness, we only consider the cases in which $\epsilon$ is strictly greater than $0$.}, denoted by $\br(\sigma_0)$, 
as:
\begin{equation*}
        \{\sigma_1 \in \Sigma_1 \mid \forall v \in V \cdot \forall \sigma'_1 \in \Sigma _1: 
        \mpinf_1(\outv(\sigma_0,\sigma_1)) > \mpinf_1(\outv(\sigma_0,\sigma'_1)) - \epsilon\}
\end{equation*}
\end{enumerate}


We note here that the definitions of best-responses can also be defined if we consider $\limsup$ instead of $\liminf$ in the mean-payoff functions.

We also introduce the following notation for zero-sum games (that are needed as intermediary steps in our algorithms). Let $\mathcal{A}$ be an arena, $v \in V$ one of its states, and $\mathcal{O} \subseteq \plays$ be a set of plays (called objective), then we write $ \mathcal{A}, v \vDash \ll i \gg \mathcal{O}$, if:
\begin{equation*}
        \exists \sigma_i\in \Sigma_i \cdot \forall \sigma_{1-i} \in \Sigma_{1-i}:
        \outv(\mathcal{A}, (\sigma_i, \sigma_{1-i})) \in \mathcal{O}, \text{for } i \in \{0,1\}
\end{equation*}
All the zero-sum games we consider in this paper are \textit{determined} meaning that for all $\mathcal{A}$, for all objectives $\mathcal{O} \subseteq \plays$ we have that
    $\mathcal{A}, v \vDash \ll i \gg \mathcal{O} \iff \mathcal{A}, v \nvDash \ll 1-i \gg \plays \setminus \mathcal{O}$.

\noindent We sometimes omit $\mathcal{A}$ when the arena being referenced is clear from the context.

\subparagraph{Convex hull and ${\sf F_{\min}}$}
Given a finite dimension $d$, a finite set $X \subset \mathbb{Q}^d$ of rational vectors, we define the convex hull $\CH{X} = \{v \mid v = \sum_{x \in X} \alpha_x\cdot x \land \forall x \in X: \alpha_x \in [0,1] \land \sum_{x \in X}\alpha_x = 1\}$ as the set of all their convex combinations. 
Let $f_{\min}(X)$ be the vector $v = (v_1, v_2, \dots , v_d)$ where $v_i = \min \{c \mid \exists x \in X: x_i = c \}$ i.e. the vector $v$ is the pointwise minimum of the vectors in $X$. 
For $S \subseteq \mathbb{Q}^d$, we define $\Fmin{S} = \{ f_{\min}(P) \mid P $ is a finite subset of $S\}$.

\subparagraph{Mean-payoffs induced by simple cycles} 
A \textit{cycle} $c$ is a sequence of edges that starts and stops in a given vertex $v$, it is simple if it does not contain repetition of any other vertex.
Given an SCC $S$, we write $\CS$ for the set of simple cycles inside $S$. Given a simple cycle $c$, for $i \in \{0,1\}$, let $\mpgen_i(c) = \frac{w_i(c)}{\mid c \mid}$ be the mean 
of the weights\footnote{We do not use $\underline{\mpgen_i}$ since $\liminf$ and $\limsup$ are the same for a finite sequence of edges.} in each dimension along the
edges in the simple cycle $c$, and we call the pair $(\mpgen_0(c), \mpgen_1(c))$ the mean-payoff coordinate of the cycle $c$. We write $\CH{\CS}$ for the convex-hull of the set of mean-payoff coordinates of simple cycles of $S$. 

\subparagraph{Adversarial Stackelberg Value for $\mpgen$} Since the set of best-responses in mean-payoff games can be empty (See Lemma 3 of \cite{FGR20}), we use the notion of $\epsilon$-best-responses for the definition of $\asvgen$ which are guaranteed to always exist\footnote{For a game $\mathcal{G}$, we also use $\asvgen_\mathcal{G}$ and $\asv_\mathcal{G}$, and drop the subscript $\mathcal{G}$ when it is clear from the context.}. 
We define
\begin{equation*}
    \asvgen(v) = \sup_{\sigma_0 \in \Sigma_0, \epsilon > 0} \inf_{\sigma_1 \in \br(\sigma_0)} \mpinf_0(\outv(\sigma_0,\sigma_1)) \text{\footnotemark}
\end{equation*}
\footnotetext{The definition of $\asvgen$, as it appears in \cite{FGR20}, is syntactically different but the two definitions are equivalent, and the one presented here is simpler.}

We also associate a (adversarial) value to a strategy $\sigma_0 \in \Sigma_0$ of Player~0, denoted
\begin{equation*}
    \asvgen(\sigma_0)(v) = \sup_{\epsilon > 0}  \inf_{\sigma_1 \in \br(\sigma_0)} \mpinf_0(\outv(\sigma_0,\sigma_1)).
\end{equation*}
Clearly, we have that
    $\asvgen(v) = \sup_{\sigma_0 \in \Sigma_0} \asvgen(\sigma_0)(v)$.
In the sequel, unless otherwise mentioned, we refer to a two-dimensional nonzero-sum two-player mean-payoff game simply as a mean-payoff game.

\subparagraph{Zero-sum case}
Zero-sum games are special cases of nonzero-sum games, where for all edges $e \in E$, we have that $w_0(e) = -w_1(e)$, i.e. the gain of one player is always equal to the opposite (the loss) of the other player.
For zero-sum games, the classical concept is the notion of (worst-case) value.
It is defined as 
\begin{equation*}
\val_\mathcal{G}(v) = \sup \limits_{\sigma_0 \in \Sigma_0} \inf \limits_{\sigma_1 \in \Sigma_1} \mpinf_0(\outv(\sigma_0, \sigma_1))
\end{equation*}
Additionally, we define the \emph{value} of a Player~0 strategy $\sigma_0$ from a vertex $v$ in a zero-sum mean-payoff game $\mathcal{G}$ as $\val_\mathcal{G}(\sigma_0)(v) = \inf \limits_{\sigma_1 \in \Sigma_1} \mpinf_0(\outv(\sigma_0, \sigma_1))$.


%% file: Robustness.tex
In this section, we study fragility and robustness properties in {\em zero-sum} and {\em nonzero-sum} games.
Additionally, we provide a notion of value, for the nonzero-sum case, that is well-suited to synthesize strategies that are robust against two types of perturbations:
\begin{itemize}
    
    \item{Modeling imprecision:} We want guarantees about the value that is obtained by a strategy in the Stackelberg game even if this strategy has been synthesized from a weighted game graph with weights that are possibly slightly wrong: small perturbations of the weight should have only limited effect on the value guaranteed by the strategy.
    
    \item{Sub-optimal responses:} We want guarantees about the value that is obtained by a strategy in the Stackelberg game even if the adversary responds with an $\epsilon$-best response instead of a perfectly optimal response (for some $\epsilon >0$):  small deviations from the best-response by the adversary should have only limited effect on the value guaranteed by the strategy.

    
\end{itemize}

\subparagraph{Formalizing deviations}
To formalize modeling imprecision, we introduce the notion of a {\em perturbed game graph}.
Given a game $\mathcal{G}$ with arena $\mathcal{A_\mathcal{G}} = (V, E, \langle V_0, V_1 \rangle, w_0, w_1)$, and a value $\delta > 0$, we write $\mathcal{G}^{\pm \delta}$ for the set $\mathcal{H}$ of games with arena $\mathcal{A_\mathcal{H}} = (V, E, \langle V_0, V_1 \rangle, w'_0, w'_1)$ where edge weight functions respect the following constraints:
        $$\forall (v_1, v_2) \in E,  \forall i \in \{0, 1 \},
        \quad w'_i (v_1, v_2) \in (w_i (v_1, v_2) + \delta, w_i (v_1, v_2) - \delta).$$
\noindent
We note that as the underlying game graph $(V,E)$ is not altered, for both players, the set of strategies in $\mathcal{G}$  is identical to the set of strategies in $\mathcal{H}$.
Finally, to formalize {\em sub-optimal responses}, we naturally use the notion of $\epsilon$-best response introduced in the previous section.

\subparagraph{Robustness in zero-sum games}
In zero-sum games, the worst-case value $\val_\mathcal{G}(\sigma_0)$ is robust against both modeling imprecision and sub-optimal responses of Player~1.
\begin{proposition}[Robustness in zero-sum games] \label{robustZS}
For all zero-sum mean-payoff games $\mathcal{G}$ with a set $V$ of vertices, for all Player~0 strategies $\sigma_0$,and for all vertices $v \in V$ we have that:
$$    \forall \delta, \epsilon > 0 : \forall \mathcal{H} \in \mathcal{G}^{\pm \delta} : 
    \inf_{\sigma_1 \in {\sf BR}^{\epsilon}_{1,\mathcal{H}}(\sigma_0)} \mpinf^\mathcal{H}_0(\outv(\sigma_0, \sigma_1)) > \val_\mathcal{G}(\sigma_0)(v) - \delta.
$$
\end{proposition}
\input{Appendix/Fragility_Robustness/robustness-zerosum}

\subparagraph{Fragility in non-zero sum games}
On the contrary, the adversarial Stackelberg value $\asvgen(\sigma_0)$ is fragile against both modeling imprecision and sub-optimal responses.
\begin{proposition}[Fragility - modeling imprecision] \label{fragile}
For all $\mu > 0$, we can construct a nonzero-sum mean-payoff game $\mathcal{G}$ and a Player~0 strategy $\sigma_0$, such that there exist $\delta > 0$, a perturbed game $\mathcal{H} \in \mathcal{G}^{\pm \delta}$, and a vertex $v$ in $\mathcal{G}$ with
    $\asvgen_\mathcal{H}(\sigma_0)(v) < \asvgen_\mathcal{G}(\sigma_0)(v) - \mu$.
\end{proposition}
\input{Appendix/Fragility_Robustness/fragile-model-imprec}
\begin{proposition}[Fragility - sub-optimal responses] \label{fragile-tolerance}
For all $\mu > 0$, we can construct a nonzero-sum mean-payoff game $\mathcal{G}$ and a Player~0 strategy $\sigma_0$, such that there exist $\epsilon > 0$ and a vertex $v$ in $\mathcal{G}$ with
    $\inf \limits_{\sigma_1 \in \br(\sigma_0)}\mpinf_0^{\mathcal{G}}(\outv(\sigma_0, \sigma_1)) < \asvgen_{\mathcal{G}}(\sigma_0)(v) - \mu$.
\end{proposition}
\input{Appendix/Fragility_Robustness/fragile-sub-optimal}
\noindent
Note that $\mu$ can be arbitrarily large and thus the adversarial Stackelberg value in the model under deviations can be arbitrarily worse than in the original model.

\subparagraph{Relation between the two types of deviations}
In nonzero-sum mean-payoff games, robustness against modeling imprecision does not imply robustness against sub-optimal responses. 
\begin{lemma}
\label{lem:model-imprec-notimplies-sub-optimal}
For all $\mu,\delta, \epsilon > 0$, we can construct a nonzero-sum mean-payoff game $\mathcal{G}$ such that for all Player~0 strategies $\sigma_0$ and vertex $v$ in $\mathcal{G}$, we have that:
\begin{equation*}
    \forall \mathcal{H} \in \mathcal{G}^{\pm \delta} : \asvgen_\mathcal{H}(\sigma_0)(v) > \inf \limits_{\sigma_1 \in \br{}_{,\mathcal{G}}} \mpinf_0^\mathcal{G} (\outv(\sigma_0, \sigma_1)) + \mu.
\end{equation*}
\end{lemma}
\input{Appendix/Fragility_Robustness/robust-not-implies-suboptimal}
\noindent
On the contrary, robustness against sub-optimal responses implies robustness against modeling imprecision.

\begin{theorem}[Robust strategy in non-zero sum games] \label{robustNZS}
For all non-zero sum mean-payoff games $\mathcal{G}$ with a set $V$ of vertices, for all $\epsilon > 0$, for all vertices $v \in V$, for all strategies $\sigma_0$ of Player~0, we have that
 $   \forall \mathcal{H} \in \mathcal{G}^{\pm \epsilon} : \asvgen_\mathcal{H}(\sigma_0)(v) > \inf \limits_{\sigma_1 \in \mathsf{BR}_{1, \mathcal{G}}^{2\epsilon}} \mpinf_0^\mathcal{G} (\outv(\sigma_0, \sigma_1)) - \epsilon.$
\end{theorem}
\input{Appendix/Fragility_Robustness/tolerance-implies-perturbation}
\noindent
We note that in the above theorem, we need to consider a strategy that is robust against $2\epsilon$-best-responses to ensure robustness against $\epsilon$ weight perturbations.

\subparagraph{\textbf{$\epsilon$-Adversarial Stackelberg Value}}
The results above suggest that, in order to obtain some robustness guarantees in nonzero-sum mean-payoff games, we must consider a solution concept 
that accounts for $\epsilon$-best responses of the adversary. This leads to the following definition: 
Given an $\epsilon > 0$, we define the adversarial value of Player~0 strategy $\sigma_0$ when Player~1 plays $\epsilon$-best-responses as
\begin{equation}
\asv(\sigma_0)(v) = \inf_{\sigma_1 \in \br(\sigma_0)} \mpinf_0(\outv(\sigma_0,\sigma_1))
\end{equation}
\noindent
and the $\epsilon$-Adversarial Stackelberg value at vertex $v$ is:
$
\asv(v)=\sup_{\sigma_0 \in \Sigma_0} \asv(\sigma_0)(v)  
$, and we note that $\asvgen(v) = \sup_{\epsilon > 0} \asv(v)$.
We can now state a theorem about combined robustness of $\asv$.

\begin{theorem}[Combined robustness of $\asv$]
\label{thm:combined-robustness}
For all nonzero-sum mean-payoff games $\mathcal{G}$ with a set $V$ of vertices, for all $\epsilon > 0$, for all $\delta > 0$, for all $\mathcal{H} \in \mathcal{G}^{\pm \delta}$, for all vertices $v \in V$, and for all strategies $\sigma_0$, we have that if $\asvgen_\mathcal{G}^{2\delta + \epsilon}(\sigma_0)(v) > c$, then for all $\mathcal{H} \in \mathcal{G}^{\pm \delta}$, we have that $\inf_{\sigma_1 \in {\sf BR}^{\epsilon}_{\mathcal{H}}(\sigma_0)} \mpinf_0^{\mathcal{H}}(\outv(\sigma_0,\sigma_1)) > c-\delta.$
\end{theorem}
\input{Appendix/Fragility_Robustness/comined-robustness}

\noindent
In the rest of the paper we study properties of $\asv$ and solve the following two problems:
\begin{itemize}
    \item{Threshold Problem of $\asv$:} Given $\mathcal{G}$, $c \in \mathbb{Q}$, an $\epsilon > 0$, and a vertex $v$, we provide a nondeterministic polynomial time algorithm to decide if $\asv(v) > c$ (see \cref{ThmNpForASV}).
    \item{Computation of $\asv$ and largest $\epsilon$:} Given $\mathcal{G}$, an $\epsilon > 0$, and a vertex $v$, we provide an exponential time algorithm to compute $\asv(v)$ (see \cref{ThmComputeASV}).
    We also establish that $\asv$ is achievable (see \cref{ThmAchiev}).
    Then we show, given a fixed threshold $c$, how to computation of largest $\epsilon$ such that $\asv(v) > c$.
    Formally, we compute $\sup \{ \epsilon > 0 \mid \asv(v) > c\}$ (See \cref{thm:ComputeEpsilon}).
\end{itemize}


%% file: Appendix/Fragility_Robustness/robustness-zerosum.tex
\begin{proof}
We first note that it comes with no surprise that considering sub-optimal responses. ($\br(\sigma_0)$) instead of optimal responses ($\mathsf{BR}_1(\sigma_0)$) of Player~1 does not decrease the \emph{value} of the game for Player~0 when she plays $\sigma_0$ as the game is zero-sum. 
Hence the value of $\sigma_0$ 
is obtained when Player~1 plays an optimal strategy or the best-response strategy. Therefore, we only need to show that all strategies of Player~0 are robust against weight perturbations. 
Formally, we must show that $\forall \delta > 0 : \forall \mathcal{H} \in \mathcal{G}^{\pm \delta} : \val_\mathcal{H}(\sigma_0)(v) > \val_\mathcal{G}(\sigma_0)(v) - \delta$.

In the case of $\delta$-perturbations, since the weight on an edge of $\mathcal{H}$ can be at most $\delta$ less than the same edge in $\mathcal{G}$, the mean-payoff over a path in $\mathcal{H}$ can be at most $\delta$ less than the mean-payoff of the same path in $\mathcal{G}$.
\end{proof}

%% file: Appendix/Fragility_Robustness/fragile-model-imprec.tex
\begin{figure}[t]
    \begin{minipage}[b]{.45\textwidth}    
    \centering
    \scalebox{0.9}{
        \begin{tikzpicture}[->,>=stealth',shorten >=1pt,auto,node distance=2cm,
                            semithick, squarednode/.style={rectangle, draw=blue!60, fill=blue!5, very thick, minimum size=7mm}]
          \tikzstyle{every state}=[fill=white,draw=black,text=black,minimum size=8mm]
        
          \node[squarednode]   (A)                    {$v_0$};
          \node[state, draw=red!60, fill=red!5]         (B) [left of=A] {$v_1$};
          \node[state, draw=red!60, fill=red!5]         (C) [right of=A] {$v_2$};
          \node[draw=none, fill=none, minimum size=0cm, node distance = 1.2cm]         (D) [below of=A] {$start$};
          \path (A) edge              node[above] {(0,0)} (B)
                    edge              node {(0,0)} (C)
                (B) edge [loop above] node {(-2$\mu$,1-$\iota/2$)} (B)
                (C) edge [loop above] node {(0,1)} (C)
                (D) edge [left] node {} (A);
        \end{tikzpicture}}
    \caption{In the original game G: $\asvgen(v_0)=0$.}
    \label{fig:fragility}
    \end{minipage}
    \qquad
    \begin{minipage}[b]{.45\textwidth}
    \centering
    \scalebox{0.9}{
    \begin{tikzpicture}[->,>=stealth',shorten >=1pt,auto,node distance=2cm,
                        semithick, squarednode/.style={rectangle, draw=blue!60, fill=blue!5, very thick, minimum size=7mm}]
      \tikzstyle{every state}=[fill=white,draw=black,text=black,minimum size=8mm]
    
      \node[squarednode]   (A)                    {$v_0$};
      \node[state, draw=red!60, fill=red!5]         (B) [left of=A] {$v_1$};
      \node[state, draw=red!60, fill=red!5]         (C) [right of=A] {$v_2$};
      \node[draw=none, fill=none, minimum size=0cm, node distance = 1.2cm]         (D) [below of=A] {$start$};
      \path (A) edge              node[above] {(0,0)} (B)
                edge              node {(0,0)} (C)
            (B) edge [loop above] node {(-2$\mu$,1)} (B)
            (C) edge [loop above] node {(0,1)} (C)
            (D) edge [left] node {} (A);
    \end{tikzpicture}}
    \caption{For $\mathcal{H} \in \mathcal{G}^{\pm \delta}$, we have $\asvgen_\mathcal{H}(v_0) \leqslant \asvgen_\mathcal{G}(v_0)-\mu$.}
    \label{fig:perturbed}
    \end{minipage}
\end{figure}
\begin{proof}
Consider the example in \cref{fig:fragility}, where we assume $0 < \iota/2 < \delta$.
Note that for all values of $\iota > 0$, we have that $\asvgen_\mathcal{G}(v_0)=0$, since Player~$1$ has no incentive to go to the left, thus the payoff of Player~$0$ is $0$ corresponding to a best-response of Player~$1$.
Thus $\asvgen_\mathcal{G}(v_0)=0$.

Now we consider the perturbed game $\mathcal{H} \in \mathcal{G}^{\pm \delta}$ in Figure \ref{fig:perturbed}.
Taking the left edge and the right edge are equally good for Player~$1$, and so the value that Player~$0$ can ensure is at most -$2 \mu < 0-\mu$.
Hence $\asvgen_\mathcal{H}(\sigma_0)(v_0) < \asvgen_\mathcal{G}(\sigma_0)(v_0)-\mu$.
\end{proof}

%% file: Appendix/Fragility_Robustness/fragile-sub-optimal.tex
\begin{proof}
We again consider the example in \cref{fig:fragility}, where we now assume $\iota > 0$.
Note that for all values of $\mu > 0$, we have that $\asvgen_{\mathcal{G}}(\sigma_0)(v_0)=0$, since Player~$1$ has no incentive to go to the left, thus the payoff of Player~$0$ is $0$ corresponding to a best-response of Player~$1$.
Thus $\asvgen_{\mathcal{G}}(\sigma_0)(v_0)=0$.
We also note that for all $\epsilon > \iota/2$ 
, we have that, in $\mathcal{G}$, playing $v_0 \to v_1$ is an $\epsilon$-best response for Player~1.
Thus, we have that $\inf \limits_{\sigma_1 \in \br(\sigma_0)}\mpinf^{\mathcal{G}}_0(\mathsf{Out}_{v_0}(\sigma_0, \sigma_1)) = -2\mu$.
Therefore, for all $\mu > 0$, we have that $\inf \limits_{\sigma_1 \in \br(\sigma_0)}\mpinf^{\mathcal{G}}_0(\mathsf{Out}_{v_0}(\sigma_0, \sigma_1)) < \asvgen_{\mathcal{G}}(\sigma_0)(v_0) - \mu$.
\end{proof}

%% file: Appendix/Fragility_Robustness/robust-not-implies-suboptimal.tex
\begin{proof}
\begin{figure}[t]
    \begin{minipage}[t]{.45\textwidth}    
    \centering
    \scalebox{0.9}{
    \begin{tikzpicture}[->,>=stealth',shorten >=1pt,auto,node distance=2cm, semithick, squarednode/.style={rectangle, draw=blue!60, fill=blue!5, very thick, minimum size=7mm}]
      \tikzstyle{every state}=[fill=white,draw=black,text=black,minimum size=8mm]
      \node[squarednode]         (A) {$v_1$};
      \node[squarednode]         (B) [right of=A] {$v_2$};
      \node[draw=none, fill=none, minimum size=0cm, node distance = 1.2cm]         (C) [left of=A] {$start$};
      \path (A) edge [bend right, below] node {$(0, 0)$} (B)
                edge [loop above] node {$(\mu', 2\delta)$} (A)
            (B) edge [loop above] node {$(0, 0)$} (B)
                edge [bend right, above] node {$(0,0)$} (A)
            (C) edge [left] node {} (A);
    \end{tikzpicture}}
    \caption{In this game $\mathcal{G}$, all vertices are controlled by Player~1. Here, $\asvgen_\mathcal{G}(v_1) = \mu'$.}
    \label{fig:perturbation-notimplies-tolerance}
    \end{minipage}
    \qquad
    \begin{minipage}[t]{.45\textwidth}    
    \centering
    \scalebox{0.9}{
    \begin{tikzpicture}[->,>=stealth',shorten >=1pt,auto,node distance=2cm, semithick, squarednode/.style={rectangle, draw=blue!60, fill=blue!5, very thick, minimum size=7mm}]
      \tikzstyle{every state}=[fill=white,draw=black,text=black,minimum size=8mm]
      \node[squarednode]         (A) {$v_1$};
      \node[squarednode]         (B) [right of=A] {$v_2$};
      \node[draw=none, fill=none, minimum size=0cm, node distance = 1.2cm]         (C) [left of=A] {$start$};
      \path (A) edge [bend right, below] node {$(0,0)$} (B)
                edge [loop above] node {$(\mu' - \iota, 2\delta - \iota)$} (A)
            (B) edge [loop above] node {$(\iota, \iota)$} (B)
                edge [bend right, above] node {$(0,0)$} (A)
            (C) edge [left] node {} (A);
    \end{tikzpicture}}
    \caption{A $\delta$-perturbed game $\mathcal{H}$ of $\mathcal{G}$ in \cref{fig:perturbation-notimplies-tolerance}. Here, we consider $0 < \iota < \delta$. Here, $\asvgen_\mathcal{H}(v_1) = \mu' - \iota$.}
    \label{fig:perturbation-notimplies-tolerance-perturbed-game}
    \end{minipage}
\end{figure}
Consider the game $\mathcal{G}$ shown in \cref{fig:perturbation-notimplies-tolerance}.
Here, since all the vertices are controlled by Player~1, the strategy of Player~0 is inconsequential.
For every $\delta > 0$, we claim that the best strategy for Player~1 across all perturbed games $\mathcal{H} \in \mathcal{G}^{\pm \delta}$ is to play $v_1 \to v_1$ forever.
One such example of a perturbed game is shown in \cref{fig:perturbation-notimplies-tolerance-perturbed-game}. 
Here, for every $0 < \iota < \delta$, we have that $v_1 \to v_1$ is the only best-response for Player~1.
Therefore, we have that $\inf \limits_{\mathcal{H} \in \mathcal{G}^{\pm \delta}} \asvgen_\mathcal{H}(\sigma_0)(v_1) = \mu' - \delta$, for all $ \delta > 0$.

However, if we relax the assumption that Player~1 plays optimally and assume that he plays an $\epsilon$-best response in the game $\mathcal{G}$, we note that Player~1 can play a strategy $(v_1^{k_1+1} v_2^{k_2+1})^\omega$, for some $k_1, k_2 \in \mathbb{N}$, such that 
$\frac{2\delta \cdot k_1}{k_1 + k_2+2} > 2\delta - \epsilon$, and 
Player~0 gets a payoff of $\frac{k_1 \cdot \mu'}{k_1 + k_2+2} > \mu'(1- \frac{\epsilon}{2\delta})$. 
Thus, we have that $\inf \limits_{\sigma_1 \in \br{}_{, \mathcal{G}}} \mpinf_0^\mathcal{G} (\outv(\sigma_0, \sigma_1)) = \mu'(1- \frac{\epsilon}{2\delta})$.
We note that the choice of $\mu'$ is arbitrary, and we can have a $\mu'$ such that $\mu' - \delta > \mu'(1-\frac{\epsilon}{2\delta}) + \mu$, i.e, we choose $\mu'$ to be large enough so that $\mu < \mu' \cdot \frac{\epsilon}{2\delta} -\delta$.
\end{proof}

%% file: Appendix/Fragility_Robustness/tolerance-implies-perturbation.tex
\begin{proof}
Consider a nonzero-sum mean-payoff game $\mathcal{G}$ and a vertex $v$ in $\mathcal{G}$ and a strategy $\sigma_0$ of Player~0. 
We let 
$\inf \limits_{\sigma_1 \in \mathsf{BR}_{1, \mathcal{G}}^{2\epsilon}} \mpinf_0^\mathcal{G} (\outv(\sigma_0, \sigma_1)) = c$
, for some $c \in \mathbb{Q}$. 
Let the supremum of the payoffs that Player~1 gets when Player~0 plays $\sigma_0$ be $y$, where $y \in \mathbb{Q}$, i.e., $\sup \{ \mpinf_1(\rho) \mid \rho \in \outv(\mathcal{G}, \sigma_0)) \} = y$.
For all outcomes $\rho$ which are in Player~1's $2\epsilon$-best response of $\sigma_0$, we have that $\mpinf_1(\rho) > y - 2\epsilon$ and $\mpinf_0(\rho) \geqslant c$.

Now, consider a game $\mathcal{H} \in \mathcal{G}^{\pm \epsilon}$ and a Player~0 strategy $\sigma_0$ played in $\mathcal{H}$. 
We can see that the maximum payoff that Player~1 gets when Player~0 plays $\sigma_0$ is bounded by $y + \epsilon$ and $y - \epsilon$, i.e., $y - \epsilon < \sup \{ \mpinf_1(\rho) \mid \rho \in \outv(\mathcal{H}, \sigma_0)) \} < y + \epsilon$. 
We let this value be denoted by $y_\mathcal{H}$. 
We note that 
if $\sup_{\rho \in \outv(\mathcal{H}, \sigma_{0})}(\mpinf_1(\rho)) = y_\mathcal{H}$, then for the corresponding play $\rho_{\mathcal{H}}$ in the game $\mathcal{G}$, the mean-payoff of Player~1 in $\rho_\mathcal{H}$ is $\mpinf_1(\rho_\mathcal{H}) > y - 2\epsilon$. 
Thus, in the game $\mathcal{G}$, we note that $\mpinf_0(\rho_\mathcal{H}) \geqslant c$ and for the corresponding play in $\mathcal{H}$, we have $\mpinf_0(\rho_\mathcal{H}) > c - \epsilon$.
Thus, we have $\asvgen_\mathcal{H}(\sigma_0)(v) > c - \epsilon = \inf \limits_{\sigma_1 \in \mathsf{BR}_{1, \mathcal{G}}^{2\epsilon}} \mpinf_0^\mathcal{G} (\outv(\sigma_0, \sigma_1)) - \epsilon$.
\end{proof}

%% file: Appendix/Fragility_Robustness/comined-robustness.tex
\begin{proof} The proof for \cref{thm:combined-robustness} is very similar to the proof of \cref{robustNZS} and involves looking at the set of $\epsilon$-best-responses in the game $\mathcal{H}$ and showing that the corresponding plays lie in the set of $(2\delta + \epsilon)$-best-responses in the game $\mathcal{G}$.
This would imply that the corresponding Player~0 mean-payoffs for the $\epsilon$-best-responses of Player~1 in every perturbed game $\mathcal{H} \in \mathcal{G}^{\pm \delta}$ would always be greater than $c - \delta$.
Therefore, we can extrapolate that $\asv_\mathcal{H}(\sigma_0)(v) > c - \delta$.
\end{proof}

%% file: DecisionAlgorithmsASV.tex
\noindent In this \mychapter, given $c \in \mathbb{Q}$, and a vertex $v$ in game $\mathcal{G}$, we study the threshold problem which is to determine if $\asv(v) > c$.

\subparagraph{Witnesses for $\asv$} For a game $\mathcal{G}$ and $\epsilon > 0$, we associate with each vertex $v$ in $\mathcal{G}$, sets of pairs or real numbers $(c, d)$ such that Player~1 has a strategy to ensure that the mean-payoffs of Player~0 and Player~1 are at most $c$ and greater than $d - \epsilon$ respectively. Formally, we have:
\begin{equation*}
\Lambda^{\epsilon}(v) = \{(c,d) \in \mathbb{R}^2 \mid v \vDash \ll 1 \gg \mpinf_0 \leqslant c \land \mpinf_1 > d-\epsilon \}.
\end{equation*}
\noindent
A vertex $v$ is $(c,d)^{\epsilon}$-bad if $(c,d) \in \Lambda^{\epsilon}(v)$. Let $c' \in \mathbb{R}$. 
A play $\pi$ in G is called a $(c',d)^{\epsilon}$-witness of $\asv(v) > c$ if $(\mpinf_0(\pi), \mpinf_1(\pi)) = (c', d)$ where $c' > c$, and $\pi$ does not contain any $(c,d)^{\epsilon}$-bad vertex. 
A play $\pi$ is called a witness for $\asv(v) > c$ if it is a $(c',d)^{\epsilon}$-witness for $\asv(v) > c$ for some $c',d$.

We first state the following technical lemma which states that if $\asv(v) > c$, then there exists a strategy $\sigma_0$ for Player~0 that enforces $\asv(\sigma_0)(v) > c$.
\input{Appendix/Fragility_Robustness/existence-strategy-threshold}

The following theorem relates the existence of a witness and the threshold problem.

\begin{theorem}
\label{ThmWitnessASVInfMem}
For all mean-payoff games $\mathcal{G}$, for all vertices $v$ in $\mathcal{G}$, for all $\epsilon > 0$, and $c \in \mathbb{Q}$, we have that $\asv(v) > c$ if and only if there exists a
$(c',d)^{\epsilon}$-witness of $\asv(v) > c$, where $d \in \mathbb{Q}$.
\end{theorem}


Towards proving the existence of the witness, 
we first recall a result from \cite{FGR20,CDEHR10} which states that
for every pair of points $(x,y)$ present in $\Fmin{\CH{\CS}}$, we can construct a play $\pi$ in the SCC $S$ such that $(\mpinf_0(\pi), \mpinf_1(\pi)) = (x,y)$. 
Further, for every play $\pi$ in SCC $S$, we have that the $(\mpinf_0(\pi), \mpinf_1(\pi))$ is present in $\Fmin{\CH{\CS}}$.

\input{Appendix/Threshold_Problem/ExistenceOfPlaysFromConvexHull}

Now we have the ingredients to prove \cref{ThmWitnessASVInfMem}.
\input{ExistenceOfASVEpsilonWitnessForIMStrategy}

Now, we establish a small witness property to show that the threshold problem is in {\sf NP}. We do this by demonstrating that the witness consists of two simple cycles $l_1$ and $l_2$ in a strongly connected component of the game graph that is reachable from $v$ such that the convex combination of the payoffs of Player~0 in these cycles exceeds the threshold $c$.

\input{Appendix/Threshold_Problem/LemPlaysAsWitnessForASV}

\input{ExpressionOfWitnessForASVEpsilon}
The following statement can be obtained by exploiting the existence of finite regular witnesses of polynomial size proved above. 
\input{NP_membership}

%% file: Appendix/Fragility_Robustness/existence-strategy-threshold.tex
\input{ExistenceOfStrategyForDecisionProblem}
\begin{proof}
\noindent The right to left direction of the proof is trivial as $\sigma_0$ can play the role of witness for 
$\sup \limits_{\sigma_0 \in \Sigma_0} \inf \limits_{\sigma_1 \in \br(\sigma_0)} \mpinf_0(\outv(\sigma_0, \sigma_1)) > c$,
i.e., if there exists a strategy $\sigma_0$ of Player $0$ such that 
$\inf \limits_{\sigma_1 \in \br(\sigma_0)} \mpinf_0(\outv(\sigma_0, \sigma_1)) > c$,
then 
$\sup \limits_{\sigma_0 \in \Sigma_0} \inf \limits_{\sigma_1 \in \br(\sigma_0)} \mpinf_0(\outv(\sigma_0, \sigma_1)) > c$.

For the left to right direction of the proof, let 
$\sup \limits_{\sigma_0 \in \Sigma_0} \inf \limits_{\sigma_1 \in \br(\sigma_0)} \mpinf_0(\outv(\sigma_0, \sigma_1)) = c'$.

\noindent By definition of $\sup$, for all $\delta > 0$, 
there exists $\sigma_0^{\delta}$ such that 
$\inf \limits_{\sigma_1 \in \br(\sigma_0^\delta)} \mpinf_0(\outv(\sigma_0^\delta, \sigma_1)) > c' - \delta$.

\noindent Let us consider a $\delta > 0$ such that $c' - \delta > c$. Such a $\delta$ exists as $c' > c$. Then we have that there exists $\sigma_0$ such that
$\inf \limits_{\sigma_1 \in \br(\sigma_0)} \mpinf_0(\outv(\sigma_0, \sigma_1)) > c' - \delta > c$.
\end{proof}

%% file: ExistenceOfStrategyForDecisionProblem.tex
\begin{lemma}
\label{lem:witness_strategy}
For all mean-payoff games $\mathcal{G}$, for all vertices $v$ in $\mathcal{G}$, for all $\epsilon > 0$, and for all rationals $c$, we have that $\mathsf{ASV}^{\epsilon}(v) > c$
iff there exists a strategy $\sigma_0 \in \Sigma_0$ such that $\mathsf{ASV}^{\epsilon}(\sigma_0)(v) > c$.
\end{lemma}

%% file: Appendix/Threshold_Problem/ExistenceOfPlaysFromConvexHull.tex
\begin{lemma}
\label{lemCHToPlay}
\textbf{\emph{(\cite{FGR20,CDEHR10})}} Let $S$ be an SCC in the arena $\mathcal{A}$ with a set $V$ of vertices, and $W$ be the maximum of the absolute values appearing on the edges in $\mathcal{A}$.
We have that
\begin{enumerate}
    \item for all $\pi \in \plays$, if $\inf(\pi) \subseteq S$, then $(\mpinf_0(\pi), \mpinf_1(\pi)) \in \Fmin{\CH{\CS}}$
    \item for all $(x,y) \in \Fmin{\CH{\CS}}$, there exists a play $\pi \in \plays$ such that $\inf(\pi) = S$ and $(\mpinf_0(\pi), \mpinf_1(\pi)) = (x,y)$.
    \item The set $\Fmin{\CH{\CS}}$ is effectively expressible in $\langle \mathbb{R}, +, < \rangle$ as a conjunction of $\mathcal{O}(m^2)$ linear inequalities, where $m$ is the number of mean-payoff coordinates of simple cycles in $S$, which is $\mathcal{O}(W \cdot |V|)$. Hence this set of inequalities can be pseudopolynomial in size.
\end{enumerate}
\end{lemma}

%% file: ExistenceOfASVEpsilonWitnessForIMStrategy.tex
\begin{proof}[Proof of \cref{ThmWitnessASVInfMem}]
In~\cite{FGR20} it has been shown that
that $\asvgen(v) > c$ if and only if there exists a witness for $\asvgen(v) > c$.
We begin by proving the right to left direction, that is, showing that the existence of a $(c',d)^\epsilon$-witness implies that $\asv(v) > c$. The proof of this direction is similar to the proof in \cite{FGR20} for the case of $\asvgen$.
\input{Appendix/Threshold_Problem/ThmWitnessASVInfMem}
We now consider the left to right direction of the proof that requires new technical tools.
We are given that $\asv(v) > c$.
From \cref{lem:witness_strategy}, we have that $\asv(v) > c$ iff there exists a strategy $\sigma_0$ of Player~0 such that $\asv(\sigma_0)(v) > c$.
Thus, there exists a $\delta > 0$, such that
\begin{equation*}
    \inf\limits_{\sigma_1 \in \br(\sigma_0)} \mpinf_0(\outv(\sigma_0, \sigma_1)) = c' = c + \delta
\end{equation*}
Let $d = \sup\limits_{\sigma_1 \in \br(\sigma_0)} \mpinf_1(\outv(\sigma_0, \sigma_1))$. We first prove that for all $\sigma_1 \in \br(\sigma_0)$, we have that $\outv(\sigma_0, \sigma_1)$ does not cross a $(c,d)^{\epsilon}$-bad vertex. For every $\sigma_1 \in \br(\sigma_0)$, we let $\pi_{\sigma_1} = \outv(\sigma_0, \sigma_1)$. We note that $\mpinf_1(\pi_{\sigma_1}) > d -\epsilon$ and $\mpinf_0(\pi_{\sigma_1}) > c$. For every $\pi' \in \outv(\sigma_0)$, we know that if $\mpinf_1(\pi') > d - \epsilon$, then there exists a strategy $\sigma_1' \in \br(\sigma_0)$ such that $\pi' = \outv(\sigma_0, \sigma_1')$. This means that $\mpinf_0(\pi') > c$. Thus we can see that every deviation from $\pi_{\sigma_1}$ either gives Player~1 a mean-payoff that is at most $d - \epsilon$ or Player~0 a mean-payoff greater than $c$. Therefore, we conclude that $\pi_{\sigma_1}$ does not cross any $(c,d)^{\epsilon}$-bad vertex. 

Now consider a sequence ($\sigma_i$)$_{i \in \mathbb{N}}$ of Player $1$ strategies such that $\sigma_i \in \br(\sigma_0)$ for all $i \in \mathbb{N}$, and $\lim \limits_{i \to \infty} \mpinf_1(\outv(\sigma_0, \sigma_i))=d$. Let $\pi_i = \outv(\sigma_0, \sigma_i)$.
Let $\inf(\pi_i)$ be the set of vertices that occur infinitely often in $\pi_i$, and let $V_{\pi_i}$ be the set of vertices appearing along the play $\pi_i$.
Since there are finitely many SCCs, w.l.o.g., we can assume that for all $i, j \in \mathbb{N}$, we have that $\inf(\pi_i) = \inf(\pi_j)$, that is, all the plays end up in the same SCC, say $S$, and also $V_{\pi_i} = V_{\pi_j}=V_{\pi}$ (say). 
Note that $S \subseteq V_{\pi}$.

Note that for every $\epsilon \ge \delta > 0$, there is a strategy $\sigma_1^\delta \in \br(\sigma_0)$ of Player $1$, and a corresponding play $\pi' = \outv(\sigma_0, \sigma_1^\delta)$ such that $\mpinf_1(\pi') > d-\delta$, and $\mpinf_0(\pi') \geqslant c'$.
Also the set $V_{\pi'}$ of vertices appearing in $\pi'$ be such that $V_{\pi'} \subseteq V_\pi$, and $\inf(\pi') \subseteq S$.

Now since $\Fmin{\CH{\CS}}$ is a closed set, we have that $(\widehat{c},d) \in \Fmin{\CH{\CS}}$ for some $\widehat{c} \geqslant c' > c$.
We now use a result from \cite{FGR20,CDEHR10} which states that
for every pair of points $(x,y)$ in $\Fmin{\CH{\CS}}$, we can construct a play $\pi$ in the SCC $S$ such that $(\mpinf_0(\pi), \mpinf_1(\pi)) = (x,y)$.
Further, for every play $\pi$ in SCC $S$, we also have that $(\mpinf_0(\pi), \mpinf_1(\pi))$ is in $\Fmin{\CH{\CS}}$.
This is formally stated as \cref{lemCHToPlay}.
Thus, we have that, there exists a play $\pi^*$ such that $(\mpinf_0(\pi^*), \mpinf_1(\pi^*)) = (\widehat{c},d)$.
Also $\inf(\pi^*) \subseteq S$, and $V_{\pi^*} \subseteq V_{\pi}$.
The proof follows since for all vertices $v \in V_\pi$, we have that $v$ is not $(c,d)^{\epsilon}$-bad.
\end{proof}

%% file: Appendix/Threshold_Problem/ThmWitnessASVInfMem.tex
We are given a play $\pi$ in $\mathcal{G}$ that starts from $v$ and the play $\pi$ is such that $(\mpinf_0(\pi), \mpinf_1(\pi)) = (c', d)$ for $c' > c$ and does not cross a $(c,d)^{\epsilon}$-bad vertex. We need to prove that $\asv(v) > c$.
We do this by defining a strategy $\sigma_0$ for Player~0, such that $\asv(\sigma_0)(v) > c$:
\begin{enumerate}
    \item $\forall h \leqslant \pi$, if $\last(h)$ is a Player~0 vertex, the strategy $\sigma_0$ is such that $\sigma_0(h)$ follows $\pi$.
    \item $\forall h \nleqslant \pi$, where there has been a deviation from $\pi$ by Player~1, we assume that Player~0 switches to a punishing strategy defined as follows: In the subgame after history $h'$ where $\last(h')$ is the first vertex from which Player~1 deviates from $\pi$, we know that Player~0 has a strategy to enforce the objective: $\mpinf_0 > c$ $\lor$ $ \mpinf_1 \leqslant d-\epsilon$. This is true because $\pi$ does not cross any $(c,d)^{\epsilon}$-bad vertex and since $n$-dimensional mean-payoff games are determined.
\end{enumerate}

Let us now establish that the strategy $\sigma_0$ satisfies $\asv(\sigma_0)(v) > c$. 
First note that, since $\mpinf_1(\pi) = d$, we have that $\sup\limits_{\sigma_1 \in \br(\sigma_0)} \mpinf_1(\outv(\sigma_0, \sigma_1)) \geqslant d$. Now consider some strategy $\sigma_1' \in \br(\sigma_0)$ and let $\pi' = \outv(\sigma_0, \sigma_1')$. Clearly, $\pi'$ is such that $\mpinf_1(\pi') > \sup\limits_{\sigma_1 \in \br(\sigma_0)} \mpinf_1(\outv(\sigma_0, \sigma_1)) -\epsilon \geqslant d-\epsilon$. If $\pi' = \pi$, we know that $\mpinf_0(\pi') > c$. If $\pi' \neq \pi$, then when $\pi'$ deviates from $\pi$ we know that Player~0 employs the punishing strategy, thus making sure that $\mpinf_0(\pi') > c$ $\lor$ $ \mpinf_1(\pi') \leqslant d-\epsilon$. Since $\sigma_1' \in \br(\sigma_0)$, it must be true that $\mpinf_0(\pi') > c$. Thus, $\forall \sigma_1' \in \br(\sigma_0)$, we have $\mpinf_0(\outv(\sigma_0, \sigma_1')) > c$. Therefore, $\asv(\sigma_0)(v) > c$, which implies $\asv(v) > c$.

%% file: Appendix/Threshold_Problem/LemPlaysAsWitnessForASV.tex
\begin{lemma}
\label{LemPlaysAsWitnessForASV}
For all mean-payoff games $\mathcal{G}$, for all vertices $v$ in $\mathcal{G}$, for all $\epsilon > 0$, and for all rationals $c$, we have that $\asv(v) > c$ if and only if there exist two simple cycles $l_1 , l_2$, three simple paths $\pi_1, \pi_2, \pi_3$ from $v$ to $l_1$, from $l_1$ to $l_2$, and from $l_2$ to $l_3$ respectively, and $\alpha, \beta \in \mathbb{Q}^{+}$, where $\alpha + \beta = 1$, such that
    \begin{enumerate}[(i)]
        \item $\alpha \cdot \mpgen_0(l_1) + \beta \cdot \mpgen_0(l_2) = c' > c$, and
        \item $\alpha \cdot \mpgen_1(l_1) + \beta \cdot \mpgen_1(l_2) = d$, for some rational $d$, and
        \item there is no $(c,d)^{\epsilon}$-bad vertex $v'$ along $\pi_1, \pi_2, \pi_3, l_1$ and $l_2$.
    \end{enumerate}
    Furthermore, $\alpha$, $\beta$, and $d$ can be chosen so that they can be represented with a polynomial number of bits.
\end{lemma}
\begin{proof}
This proof is similar to the proof of Lemma 8 in \cite{FGR20}.
For the right to left direction of the proof, where we are given finite acyclic plays $\pi_1, \pi_2, \pi_3$, simple cycles $l_1$ and $l_2$ and constants $\alpha, \beta$, we consider the witness $\pi = \pi_1\rho_1\rho_2\rho_3\dots$ where, for all $i \in \mathbb{N}$, we let $\rho_i = l_1^{[\alpha \cdot i]}.\pi_2.l_2^{[\beta \cdot i]}.\pi_3$. 
We know that $\mpinf_1(\pi) = \alpha \cdot \mpgen_1(l_1) + \beta \cdot \mpgen_1(l_2) = d$ and $\mpinf_0(\pi) = \alpha \cdot \mpgen_0(l_1) + \beta \cdot \mpgen_0(l_2) > c$. For all vertices $v$ in $\pi_1, \pi_2, \pi_3, l_1$ and $l_2$, it is given that $v$ is not $(c,d)^\epsilon$-bad. 
Therefore, $\pi$ is a suitable witness thus proving from \cref{ThmWitnessASVInfMem} that $\asv(v) > c$.

For the left to right direction of the proof, we are given $\asv(v) > c$. 
Using \cref{ThmWitnessASVInfMem}, we can construct a play $\pi$ such that $\mpinf_0(\pi) > c$ and $\mpinf_1(\pi) = d$, and $\pi$ does not cross a $(c,d)^{\epsilon}$-bad vertex, i.e., for all vertices $v'$ appearing in $\pi$, we have that $v' \nvDash \ll 1 \gg \mpinf_0 \leqslant c \land \mpinf_1 > d-\epsilon$. 
First, the value $d$ must be chosen so that the vertices $v'$ appearing in $\pi$ do not belong to the $(c,d)^{\epsilon}$-bad vertices.
Let $\inf(\pi) = S$ be the set of vertices appearing infinitely often in $\pi$. Note that $S$ forms an SCC.
By abuse of notation, we also denote this SCC by $S$ here.
By \cref{lemCHToPlay}, we have that $(\mpinf_0(\pi), \mpinf_1(\pi)) \in \Fmin{\CH{\CS}}$.
From Proposition 1 of \cite{CDEHR10}, for a bi-weighted arena, we have that $\Fmin{\CH{\CS}} = \CH{\Fmin{\CS}}$.
Since $\CH{\Fmin{\CS}}$ can be expressed using conjunctions of linear inequalities whose coefficients have polynomial number of bits, the same also follows for $\Fmin{\CH{\CS}}$ in a bi-weighted arena.
In addition, it is proven in~\cite{BR15} that the set $\Lambda^{\epsilon}(v')$ is definable by a disjunction of conjunctions of linear inequalities whose coefficients have polynomial number of bits in the descriptions of the game $\mathcal{G}$ and of $\epsilon$.
Hence $\overline{\Lambda}^{\epsilon}(v')$ is also definable by a disjunction of conjunctions of linear inequalities whose coefficients have polynomial number of bits in the descriptions of the game $\mathcal{G}$ and of $\epsilon$.
As a consequence of 
Theorem~2 in~\cite{BR15} which states that given a system of linear inequalities that is satisfiable, there exists a point with polynomial representation that satisfies the system, we have that $d$ can be chosen such that $(c,d) \in \overline{\Lambda}^{\epsilon}(v')$ and $(\mpinf_0(\pi), d) \in \Fmin{\CH{\CS}}$, and hence $d$ can be represented with a polynomial number of bits.

Second, by applying the Carath\'{e}odory baricenter theorem, we can find two simple cycles $l_1, l_2$ in the SCC $S$ and acyclic finite plays $\pi_1, \pi_2 $ and $\pi_3$ from $\pi$, and two positive rational constants $\alpha, \beta \in \mathbb{Q^+}$, such that \first($\pi_1$) $ = v$, \first($\pi_2$) = \last($\pi_1$), \first($\pi_3$) = \last($\pi_2$), \first($\pi_2$) = \last($\pi_3$), and \first($\pi_2$) = \first($l_1$), and \first($\pi_3$) = \first($l_2$), and $\alpha + \beta = 1$, $\alpha \cdot \mpgen_0(l_1) + \beta \cdot \mpgen_0(l_2) > c$ and $\alpha \cdot \mpgen_1(l_1) + \beta \cdot \mpgen_1(l_2) = d$. 
Again, using Theorem~2 in
~\cite{BR15}, 
we can assume that $\alpha$ and $\beta$ are rational values that can be represented using a polynomial number of bits.
We note that for all vertices $v$ in $\pi_1, \pi_2, \pi_3, l_1$ and $l_2$, we have that $v$ is not $(c,d)^{\epsilon}$-bad.
\end{proof}

%% file: ExpressionOfWitnessForASVEpsilon.tex
A play $\pi$ is called a \textit{regular}-witness of $\asv(v) > c$ if it is a witness of $\asv(v) > c$ and can be expressed as $\pi = u \cdot v^\omega$, where $u$ is a prefix of a play, and where $v$ is a finite sequence of edges.
We prove in that following theorem there exists a \emph{regular} witness for $\asv(v) > c$. 
The existence of a regular witness helps in the construction of a finite memory strategy for Player~0.
\begin{theorem}
\label{ThmWitnessASVFinMem}
For all mean-payoff games $\mathcal{G}$, for all vertices $v$ in $\mathcal{G}$, for all $\epsilon > 0$, and for all rationals $c$, we have that $\asv(v) > c$ if and only if there exists a regular $(c',d)^{\epsilon}$-witness of $\asv(v) > c$, where $d$ is some rational.
\end{theorem}
\begin{proof}
\input{Appendix/Threshold_Problem/ThmWitnessASVFinMem}
\end{proof}

%% file: Appendix/Threshold_Problem/ThmWitnessASVFinMem.tex
Consider the witness $\pi$ in the proof of \cref{LemPlaysAsWitnessForASV}.
We construct a regular-witness $\pi'$ for $\asv(v) > c$ where $\pi' = \pi_1.(l_1^{[\alpha' \cdot k]}.\pi_2.l_2^{[\beta' \cdot k]}.\pi_3)^{\omega}$ and $\alpha', \beta'$ are constants in $\mathbb{Q}$ and $k$ is some large integer. We construct $\pi'$ by modifying $\pi$ as follows.
We need to consider the following cases.
\begin{caseof}
    \case{$\mpgen_0(l_1) > \mpgen_0(l_2)$ and $\mpgen_1(l_1) < \mpgen_1(l_2)$}
    Here, one simple cycle, $l_1$, increases Player~0's mean-payoff while the other simple cycle, $l_2$, increases Player~1's mean-payoff. We can build a witness $\pi' = \pi_1.(l_1^{[\alpha \cdot k]}.\pi_2.l_2^{[\beta \cdot (k+\tau)]}.\pi_3)^{\omega}$ for some very large $k \in \mathbb{N}$ and for some small $\tau > 0$ such that $\mpinf_0(\pi') > c$ and $\mpinf_1(\pi') = d$.\footnote{For more details, we refer the reader to the \cref{app:expressionForKAndTau}.}
    We note that $k$ and $\tau$ are polynomial in the size of $\mathcal{G}$, and the largest weight $W$ appearing on the edges of $\mathcal{G}$.
    \case{$\mpgen_0(l_1) < \mpgen_0(l_2)$ and $\mpgen_1(l_1) > \mpgen_1(l_2)$}
    This is analogous to \textbf{case 1}, and proceeds as mentioned above.
    \case{$\mpgen_0(l_1) > \mpgen_0(l_2)$ and $\mpgen_1(l_1) > \mpgen_1(l_2)$}
    One cycle, $l_1$, increases both Player~0 and Player~1's mean-payoffs, while the other, $l_1$, decreases it. In this case, we can just omit one of the cycles 
    and consider the one that
    gives a larger mean-payoff, to get a finite memory strategy. Thus, $\pi' = \pi_1.l_1^{\omega}$ and we get $\mpinf_0(\pi') > c$ , $\mpinf_1(\pi') \geqslant d$.
    Suppose $\mpinf_1(\pi') =d' \geqslant d$.
    Since no vertex in $\pi_1$, $\pi_2$, $\pi_3$, $l_1$, and $l_2$ is $(c,d)^{\epsilon}$-bad, we also have that they are not $(c,d')^{\epsilon}$-bad, and thus $\pi'$ is a witness for $\asv(v) > c$.
    \case{$\mpgen_0(l_1) < \mpgen_0(l_2)$ and $\mpgen_1(l_1) < \mpgen_1(l_2)$}
    This is analogous to \textbf{case 3}, and proceeds as mentioned above.
\end{caseof}
In each of these cases, we 
have that
$\asv(v) > c$: $\mpinf_0(\pi') > c$ and $\mpinf_1(\pi') \geqslant d$.
Since we know that $\pi$ does not cross a $(c,d)^{\epsilon}$-bad vertex, 
and the vertices of the play $\pi'$ are a subset of the vertices of the play $\pi$, we have that $\pi'$ is a witness for $\asv(v) > c$.

%% file: NP_membership.tex
\begin{theorem}
    \label{ThmNpForASV}
        For all mean-payoff games $\mathcal{G}$, for all vertices $v$ in $\mathcal{G}$, for all $\epsilon > 0$, and for all $c \in \mathbb{Q}$, it can be decided in nondeterministic polynomial time if $\asv(v) > c$, and a 
    pseudopolynomial memory strategy of Player~0 suffices for this threshold.
    Furthermore, this decision problem is at least as hard as solving zero-sum mean-payoff games.
\end{theorem}
\input{Appendix/Threshold_Problem/ThmNPForASV}

%% file: Appendix/Threshold_Problem/ThmNPForASV.tex
\noindent 
For proving \cref{ThmNpForASV}, we start by stating a property of multi-dimensional mean-payoff games proved in \cite{VCDHRR15} that we rephrase here for a two-dimensional mean-payoff game. This property expresses a relation between mean-payoff $\limsup$ and mean-payoff $\liminf$ objectives.
We recall that in \cite{VCDHRR15}, the objective of Player~1 is to maximize the payoff in each dimension, i.e., for two-dimensional setting, given two rational $c$ and $d$, Player~1 wins if he has a winning strategy for $\mpinf_0 \geqslant c \land \mpinf_1 \geqslant d$; otherwise Player~0 wins due to determinacy of multi-dimensional mean-payoff games.
We call the mean-payoff game setting in \cite{VCDHRR15} {\sf 2D-max} mean-payoff games to distinguish it from the mean-payoff games that we consider here.
Later we will relate the two settings.
\begin{proposition}
    \label{PropInfEqSup}
    \textbf{\emph{(Lemma 14 in \cite{VCDHRR15})}} For all mean-payoff games $\mathcal{G}$, for all vertices $v$ in the game $\mathcal{G}$, and for all rationals $c, d$, we have
    \begin{equation*}
        v \vDash \ll 1 \gg \overline{\mpgen}_0 \geqslant c \land \mpinf_1 \geqslant d
    \end{equation*}
    if and only if
    \begin{equation*}
        v \vDash \ll 1 \gg \mpinf_0 \geqslant c \land \mpinf_1 \geqslant d
    \end{equation*}    
\end{proposition}

We now recall another property of multi-dimensional mean-payoff games proved in \cite{VCDHRR15} that we rephrase here for a {\sf 2D-max} mean-payoff game. This property expresses a bound on the weight of every finite play $\pi^f \in \outv(\sigma_0)$ where $\sigma_0$ is a memoryless winning strategy for Player~0.
\begin{lemma}{\textbf{\emph{(Lemma 10 in \cite{VCDHRR15})}}}
    \label{LemWeightPlayGrtThanC}
    For all {\sf 2D-max} mean-payoff games $\mathcal{G}$, for all vertices $v$ in $\mathcal{G}$, and for all rationals $c, d$, if Player~0 \footnote{Player~0 is called Player~2 in \cite{VCDHRR15}.} wins $\mpinf_0 < c \lor \mpinf_1 < d$ from $v$ then she has a memoryless winning strategy $\sigma_0$ to do so, and there exist three constants $m_\mathcal{G}, c_\mathcal{G}, d_\mathcal{G} \in \mathbb{R}$ such that: $c_\mathcal{G} < c, d_\mathcal{G} < d$, and for all finite plays $\pi^f \in \outv(\sigma_0)$, i.e. starting from $v$ and compatible with $\sigma_0$, we have that
    \begin{equation*}
        w_0(\pi^f) \leqslant m_\mathcal{G} + c_\mathcal{G} \cdot |\pi^f|
    \end{equation*}
    or
    \begin{equation*}
        w_1(\pi^f) \leqslant m_\mathcal{G} + d_\mathcal{G} \cdot |\pi^f|
    \end{equation*}
\end{lemma}

We now relate the {\sf 2D-max} mean-payoff game in \cite{VCDHRR15} where the objective of Player~1 is to maximize the payoff in both dimensions to our setting where in a game $\mathcal{G}$, Player~1 maximizes the payoff in the second dimension, and minimizes the payoff on the first dimension from his set of available responses to a strategy of Player~0. The objective of Player~0 then is to maximize the payoff in the first dimension and minimize the payoff in the second dimension, i.e. given two rationals $c$ and $d$, Player~0's objective is to ensure $v \vDash \ll 0 \gg \mpinf_0 > c \lor \mpinf_1 < d$. We now state a modification of \cref{LemWeightPlayGrtThanC} as follows:
\begin{lemma}
    \label{LemWeightPlayLessThanC}
    For all mean-payoff games $\mathcal{G}$, for all vertices $v$ in $\mathcal{G}$, and for all rationals $c, d$, if Player~0 wins $\mpinf_0 > c \lor \mpinf_1 < d$ from $v$, then she has a memoryless winning strategy $\sigma_0$ to do so, and there exist three constants $m_\mathcal{G}, c_\mathcal{G}, d_\mathcal{G} \in \mathbb{R}$ such that $c_\mathcal{G} > c, d_\mathcal{G} < d$, and for all finite plays $\pi^f \in \outv(\sigma_0)$, i.e. starting in $v$ and compatible with $\sigma_0$, we have that
    \begin{equation*}
        w_0(\pi^f) \geqslant -m_\mathcal{G} + c_\mathcal{G} \cdot |\pi^f|
    \end{equation*}
    or
    \begin{equation*}
        w_1(\pi^f) \leqslant m_\mathcal{G} + d_\mathcal{G} \cdot |\pi^f|
    \end{equation*}
\end{lemma}
\begin{proof}
We show this by a reduction to a {\sf 2D-max} mean-payoff game where Player~0's objective is to ensure $v \vDash \ll 0 \gg \mpinf_0 < c \lor \mpinf_1 < d$ in a {\sf 2D-max} mean-payoff game.

We prove this lemma in two parts.
If Player~0 wins $\mpinf_0 > c \lor \mpinf_1 < d$, we show
\begin{inparaenum}[(i)]
\item \label{part1}  the existence of a memoryless strategy $\sigma_0$ for Player~0, and
\item \label{part2} that there exist three constants $m_\mathcal{G}, c_\mathcal{G}, d_\mathcal{G} \in \mathbb{R}$ such that $c_\mathcal{G} > c, d_\mathcal{G} < d$, and for all finite plays $\pi^f \in \outv(\sigma_0)$, i.e. starting from $v$ and compatible with $\sigma_0$, we have that either $w_0(\pi^f) \geqslant -m_\mathcal{G} + c_\mathcal{G} \cdot |\pi^f|$ or $w_1(\pi^f) \leqslant m_\mathcal{G} + d_\mathcal{G} \cdot |\pi^f|$.
\end{inparaenum}

Assume that Player~0 has a winning strategy from vertex $v$ in $\mathcal{G}$ for $\mpinf_0 > c \lor \mpinf_1 < d$.
To prove (\ref{part1}), we subtract $2c$ from the weights on the first dimension of all the edges, followed by multiplying them with -$1$.
We call the resultant {\sf 2D-max} mean-payoff game $\mathcal{G}'$, and we have that $v \vDash \ll 0 \gg \mpinf_0 > c \lor \mpinf_1 < d$ in $\mathcal{G}$ if and only if $v \vDash \ll 0 \gg \overline{\mpgen}_0 < c \lor \mpinf_1 < d$ in $\mathcal{G}'$.
Using \cref{PropInfEqSup} and determinacy of multi-dimensional mean-payoff games, it follows that $v \vDash \ll 0 \gg \mpinf_0 > c \lor \mpinf_1 < d$ in $\mathcal{G}$ if and only if $v \vDash \ll 0 \gg \mpinf_0 < c \lor \mpinf_1 < d$ in $\mathcal{G}'$.
Also from \cite{VCDHRR15}, we have that if Player~0 has a winning strategy in $\mathcal{G}'$ for $\mpinf_0 < c \lor \mpinf_1 < d$, then she has a memoryless strategy $\sigma_0$ for the same, and the proof of Lemma 14 in \cite{VCDHRR15} shows that same memoryless strategy $\sigma_0$ is also winning for $\overline{\mpgen}_0 < c \lor \mpinf_1 < d$, thus concluding that if Player~0 wins $\mpinf_0 > c \lor \mpinf_1 < d$ in $\mathcal{G}$ from vertex $v$, then she has a memoryless winning strategy.

We prove (\ref{part2}) by contradiction.
Assume that Player~0 wins $\mpinf_0 > c \lor \mpinf_1 < d$ from vertex $v$ in $\mathcal{G}$, and by part (\ref{part1}), she has a memoryless winning strategy $\sigma_0$.
Assume for contradiction, that there does not exist three constants $m_\mathcal{G}, c_\mathcal{G}, d_\mathcal{G} \in \mathbb{R}$ such that $c_\mathcal{G} > c, d_\mathcal{G} < d$, such that for all finite plays $\pi^f \in \outv(\sigma_0)$, i.e. starting in $v$ and compatible with $\sigma_0$, we have either $w_0(\pi^f) \geqslant -m_\mathcal{G} + c_\mathcal{G} \cdot |\pi^f|$ or $w_1(\pi^f) \leqslant m_\mathcal{G} + d_\mathcal{G} \cdot |\pi^f|$.

Consider the steps in the construction of $\mathcal{G}'$ as defined above.
As we subtract $2c$ from the weights on the first dimension of each edge, and multiply the resultant weights on the first dimension by -$1$, we have that there does not exist three constants $m_\mathcal{G}, c_\mathcal{G}, d_\mathcal{G} \in \mathbb{R}$ such that $c_\mathcal{G} > c, d_\mathcal{G} < d$, and for all finite plays $\pi^f \in \outv(\sigma_0)$ in $\mathcal{G}$, i.e. starting from $v$ and compatible with $\sigma_0$, we have either 
    \begin{equation*}
        w_0(\pi^f) \geqslant -m_\mathcal{G} + c_\mathcal{G} \cdot |\pi^f|
    \end{equation*}
    or
    \begin{equation*}
        w_1(\pi^f) \leqslant m_\mathcal{G} + d_\mathcal{G} \cdot |\pi^f|
    \end{equation*}
if and only if there does not exist three constants $m_\mathcal{G}, c_\mathcal{G}, d_\mathcal{G} \in \mathbb{R}$ such that $c_\mathcal{G} > c, d_\mathcal{G} < d$, and for all finite plays $\pi^f \in \outv(\sigma_0)$ in the {\sf 2D-max} mean-payoff game $\mathcal{G}'$ for the objective $\mpinf_0 < c \lor \mpinf_1 < d$, i.e. starting in $v$ and compatible with $\sigma_0$, we have either 
    \begin{equation*}
        w_0(\pi^f) \leqslant m_\mathcal{G} + (2c-c_\mathcal{G}) \cdot |\pi^f|
    \end{equation*}
    or
    \begin{equation*}
        w_1(\pi^f) \leqslant m_\mathcal{G} + d_\mathcal{G} \cdot |\pi^f|
    \end{equation*}
Let $2c-c_\mathcal{G} = c'_\mathcal{G}$, and we have that $c'_\mathcal{G} < c$.
Now since $\sigma_0$ is a winning for Player~0 for the objective $\mpinf_0 < c \lor \mpinf_1 < d$ in $\mathcal{G}'$ from $v$, we reach a contradiction by \cref{LemWeightPlayGrtThanC}, and due to determinacy of multi-dimensional mean-payoff games.
\end{proof}
\noindent 
Using \cref{LemWeightPlayLessThanC}, we can now prove that Player~0 can ensure from a vertex $v$ that 
$v \nvDash \ll 1 \gg \mpinf_0 \leqslant c \land \mpinf_1 > d$ if and only if she can also ensure that $v \nvDash \ll 1 \gg \mpinf_0 \leqslant c \land \mpinf_1 \geqslant d'$ for all $d' > d$.
This is established in the following lemma.
\begin{lemma}
    \label{ConjGrtIsGrtEq}
    For all mean-payoff games $\mathcal{G}$, for all vertices $v \in \mathcal{G}$, and for all rationals $c, d$, we have that:
    \begin{equation*}
        v \vDash \ll 1 \gg \mpinf_0 \leqslant c \land \mpinf_1 > d
    \end{equation*}
    if and only if there exists a $d' \in \mathbb{R}$, where $d' > d$ such that
    \begin{equation*}
        v \vDash \ll 1 \gg \mpinf_0 \leqslant c \land \mpinf_1 \geqslant d'
    \end{equation*}
\end{lemma}

\begin{proof}
    For the right to left direction of the proof, it is trivial to see that 
    if $v \vDash \ll 1 \gg \mpinf_0 \leqslant c \land \mpinf_1 \geqslant d'$ for some $d' > d$, then we have that $v \vDash \ll 1 \gg \mpinf_0 \leqslant c \land \mpinf_1 > d$.

    For the left to right direction of the proof, we prove the contrapositive, i.e., we assume that $\forall d' > d$, we have $v \nvDash \ll 1 \gg \mpinf_0 \leqslant c \land \mpinf_1 \geqslant d'$. Now we prove that $v \nvDash \ll 1 \gg \mpinf_0 \leqslant c \land \mpinf_1 > d$.

    Since $\forall d'> d$, Player~1 loses $\mpinf_0 \leqslant c \land \mpinf_1 \geqslant d'$ from a given vertex $v$, due to determinacy of multi-dimensional mean-payoff games, Player~0 wins $\mpinf_0 > c \lor \mpinf_1 < d'$ from vertex $v$. By \cref{LemWeightPlayLessThanC}, Player~0 has a memoryless strategy $\sigma_0$ to achieve the objective $\mpinf_0 > c \lor \mpinf_1 < d'$ from vertex $v$. Note that Player~0 has only finitely many memoryless strategies. Therefore there exists a strategy $\sigma_0^*$ that achieves the objective $v \vDash \ll 0 \gg \mpinf_0 > c \lor \mpinf_1 < d'$ for all $d' > d$. 
    Now from \cref{LemWeightPlayGrtThanC}, for every $d' > d$, there exists three constants $m_\mathcal{G}, c_\mathcal{G}, d'_\mathcal{G} \in \mathbb{R}$ such that $c_\mathcal{G} > c, d'_\mathcal{G} < d'$, and for all finite plays $\pi^f \in \outv(\sigma_0^*)$, we have that
    \begin{equation*}
        w_0(\pi^f) \geqslant -m_\mathcal{G} + c_\mathcal{G} \cdot |\pi^f|
    \end{equation*}
    or
    \begin{equation*}
        w_1(\pi^f) \leqslant m_\mathcal{G} + d'_\mathcal{G} \cdot |\pi^f|
    \end{equation*}
    
    Note that since the above is true for every $d'>d$, we can indeed consider a $d_\mathcal{G} \in \mathbb{R}$, where $d_\mathcal{G} \leqslant d$, such that for all $d' > d$, and for all finite plays $\pi^f \in \outv(\sigma_0^*)$, we have that
    \begin{equation*}
        w_0(\pi^f) \geqslant -m_\mathcal{G} + c_\mathcal{G} \cdot |\pi^f|
    \end{equation*}
    or we have that
    \begin{equation*}
        w_1(\pi^f) \leqslant m_\mathcal{G} + d_\mathcal{G} \cdot |\pi^f|
    \end{equation*}

    Hence, for every play $\pi \in \outv(\sigma_0^*)$, we
    have that
    \begin{equation*}
        \overline{\mpgen}_0(\pi) \geqslant c_\mathcal{G} \lor \mpinf_1(\pi) \leqslant d_\mathcal{G}
    \end{equation*}

    Thus, we get 
    \begin{align*}
        v & \vDash \ll 0 \gg \overline{\mpgen}_0 \geqslant c_\mathcal{G} \lor \mpinf_1 \leqslant d_\mathcal{G} && \\
        \iff v & \vDash \ll 0 \gg \overline{\mpgen}_0 > c \lor \mpinf_1 < d' && \text{for every $d' > d$,}
        \text{ since $c_\mathcal{G} > c$ and } d_\mathcal{G} < d'.
    \end{align*}
    
    We now construct a {\sf 2D-max} mean-payoff game $\mathcal{G'}$ from the given game $\mathcal{G}$ by multiplying the first dimension of the weights of all the edges by $-1$. Thus, in the game $\mathcal{G'}$, we get
    \begin{align*}
        v & \vDash \ll 0 \gg \mpinf_0 < -c \lor \mpinf_1 < d' && \text{for every $d' > d$} \\
        \iff v & \vDash \ll 0 \gg \overline{\mpgen}_0 < -c \lor \mpinf_1 < d' && \text{for every $d' > d$} \text{ [from \cref{PropInfEqSup}]}
    \end{align*}
    
    We now construct a game $\mathcal{G''}$ from the game $\mathcal{G'}$ by multiplying the first dimension of the weights of all the edges by -$1$. Note that, we get back the original game $\mathcal{G}$ after this modification, i.e. 
    $\mathcal{G}''$ has the same arena as that of $\mathcal{G}$.
    Thus, in the game $\mathcal{G}$, we have that 
    \begin{align*}
        \hspace{0.29in} v & \vDash \ll 0 \gg \mpinf_0 > c \lor \mpinf_1 < d' &&  \text{for every $d'> d$ \hspace{1.44in}}
    \end{align*}
    Recall by \cref{LemWeightPlayLessThanC}, if Player~0 wins $\mpinf_0 > c \lor \mpinf_1 < d'$, then she has a memoryless strategy for this objective, and since there are finitely many memoryless strategies, there exists a memoryless strategy $\sigma_0^*$ of Player $0$ that wins for all $d' > d$.
    This also implies that by using $\sigma_0^*$, from vertex $v$, Player~0 can ensure $\mpinf_0 > c \lor \mpinf_1 \leqslant d$, that is,
    \begin{align*}
        v \vDash & \ll 0 \gg \mpinf_0 > c \lor \mpinf_1 \leqslant d && \\
        \iff v \nvDash & \ll 1 \gg \mpinf_0 \leqslant c \land \mpinf_1 > d && \text{[By determinacy of multi-dimensional} \\
        & && \text{mean-payoff games]}
    \end{align*}
\end{proof}

\noindent We are finally ready to prove \cref{ThmNpForASV}.
\begin{proof}[Proof of \cref{ThmNpForASV}]
According to \cref{LemPlaysAsWitnessForASV}, we consider a nondeterministic Turing machine that establishes the membership to {\sf NP} by guessing a reachable SCC $S$, a finite play $\pi_1$ to reach $S$ from $v$, two simple cycles $l_1, l_2$, along with the finite play $\pi_1$ from $v$ to $l_1$, and the two finite plays $\pi_2$ and $\pi_3$ that connects the two simple cycles, and parameters $\alpha, \beta \in \mathbb{Q}^{+}$. 
Additionally, for each vertex $v'$ that appear along the plays $\pi_1, \pi_2$ and $\pi_3$, and on the simple cycles $l_1$ and $l_2$, the Turing machine guesses a memoryless strategy $\sigma_0^{v'}$ for Player~0 that establishes $v' \nvDash \ll 1 \gg \mpinf_0 \leqslant c \land \mpinf_1 > d - \epsilon$ which implies by determinacy of multi-dimensional mean-payoff games, that $v' \vDash \ll 0 \gg \mpinf_0 > c \lor \mpinf_1 \leqslant d - \epsilon$.

Besides, from \cref{ThmWitnessASVFinMem}, we can obtain a regular witness $\pi'$. Using $\pi'$, we build a finite memory strategy $\sigma_0^{\sf FM}$ for Player~0 as stated below:
\begin{enumerate}
    \item Player~0 follows $\pi'$ if Player~1 does not deviate from $\pi'$. The finite memory strategy stems from the finite $k$ as required in the proof of \cref{ThmWitnessASVFinMem}.
    \item For each vertex $v' \in \pi'$, Player~0 employs the memoryless strategy $\sigma_0^{v'}$ that establishes $v' \nvDash \ll 1 \gg \mpinf_0 \leqslant c \land \mpinf_1 > d-\epsilon$.
    The existence of such a memoryless strategy follows from the proof of \cref{ConjGrtIsGrtEq}.
\end{enumerate}

It remains to show that all the guesses can be verified in polynomial time. The only difficult part concerns the memoryless strategies of Player~0 to punish deviations of Player~1 from the witness play $\pi'$. These memoryless strategies are used to prove that the witness does not cross $(c,d)^{\epsilon}$-bad vertices. 
For vertex $v'\in \pi'$, we consider a memoryless strategy $\sigma_0^{v'}$, and we need to establish that it can enforce $\mpinf_0 > c \lor \mpinf_1 \leqslant d - \epsilon$.
Towards this, we adapt the proof of Lemma 10 in \cite{VCDHRR15}, which in turn is based on the polynomial time algorithm of Kosaraju and Sullivan~\cite{KS88} for detecting zero-cycles in multi-weighted directed graphs.

Consider the bi-weighted graph obtained from $\mathcal{G}$ by fixing the choices of Player~0 according to the memoryless strategy $\sigma_0^{v'}$.
We first compute the set of maximal SCCs that are reachable from $v'$ in this bi-weighted graph.
This can be done in linear time.
For each SCC $S$, we need to check that Player~1 cannot achieve $\mpinf_0 \leqslant c \land \mpinf_1 > d - \epsilon$.

We first recall the definition of multi-cycles from \cite{VCDHRR15}, which is a multi-set of simple cycles from the SCC $S$.
For a simple cycle $C = (e_1, \dots, e_n)$, let $w(c) = \sum_{e \in C} w(e)$.
For a multi-cycle $\mathcal{C}$, let $w(\mathcal{C}) = \sum_{C \in \mathcal{C}} w(C)$ (note that in this summation, a cycle $C$ may appear multiple times in $\mathcal{C}$).
A non-negative multi-cycle is a non-empty multi-set of simple cycles $\mathcal{C}$ such that $w(\mathcal{C}) \geqslant 0$ (i.e., in both the dimensions, the weight is non-negative).
In \cite{VCDHRR15}, it has been shown that the problem of deciding if $S$ has a non-negative multi-cycle can be solved in polynomial time by solving a set of linear inequalities.
In our case, we are interested in multi-cycles such that $w_0(\mathcal{C}) \leqslant c$ and $w_1(\mathcal{C}) > d-\epsilon$.
As in the proof of Lemma 10 in \cite{VCDHRR15}, this can be checked by defining the following set of linear constraints.
Let $V_S$ and $E_S$ respectively denote the set of vertices and the set of edges in $S$.
For every edge $e \in E_S$, we consider a variable $\chi_e$.
\begin{enumerate}[(a)]
    \item For $v \in V_S$, let $\mathsf{In}(v)$ and $\mathsf{Out}(v)$ respectively denote the set of incoming edges to $v$ and the set of outgoing edges from $v$.
    For every $v \in V_S$, we define the linear constraint $\sum_{e \in \mathsf{In}(v)} \chi_e = \sum_{e \in \mathsf{Out}(v)} \chi_e$ which intuitively models flow constraints.
    \item For every $e \in E_S$, we define the constraint $\chi_e \geqslant 0$.
    \item We also add the constraint $\sum_{e \in E_S} \chi_e \cdot w_0(e) \leqslant c$ and $\sum_{e \in E_S} \chi_e \cdot w_1(e) > d-\epsilon$.
    \item Finally, we define the constraint $\sum_{e \in E_S} \chi_e \geqslant 1$ that ensures that the multi-cycle is non-empty.
\end{enumerate}
This set of linear constraints can be solved in polynomial time, and formally following the arguments from \cite{KS88}, it has a solution if and only if there exists a multi-cycle $\mathcal{C}$ such that $w_0(\mathcal{C}) \leqslant c$ and $w_1(\mathcal{C}) > d-\epsilon$.
The {\sf NP}-membership follows since we have linearly many maximal SCCs from each vertex $v'$ in the bi-weighted graph that is obtained from $\mathcal{G}$ by fixing the choices of Player~0 according to the memoryless strategy $\sigma_0^{v'}$, and there are linearly many vertices $v'$ for which we need to check that $v'$ is not $(c,d)^\epsilon$-bad.

Now we show that the memory required by the strategy $\sigma_0^{\sf FM}$ as described above is pseudopolynomial in the input size.
Recall from the proof of \cref{ThmWitnessASVFinMem} that $k$ and $\tau$ are polynomial in the size of $\mathcal{G}$, and the largest weight $W$ appearing on the edges of $\mathcal{G}$.
Assuming that the weights are given in binary, 
the number of states in the finite state machine realizing this strategy is thus $\poly(|\mathcal{G}|, W)$, and hence pseudopolynomial in the input size, assuming that the weights are given in binary.

\input{Appendix/Threshold_Problem/ThmHardnessThreshold}
\end{proof}

%% file: Appendix/Threshold_Problem/ThmHardnessThreshold.tex
Now we prove that the threshold problem is at least as hard as solving zero-sum mean-payoff games.
We show the proof for $\asv(v) > c$.
The proof for the case of $\asvgen(v) > c$ is exactly the same.
Consider a zero-sum mean-payoff game $\mathcal{G}_0 = (\mathcal{A}, \mpgen)$, where $\mathcal{A} = (V,E, \zug{V_0,V_1}, w)$.
We construct a bi-weighted mean-payoff game $\mathcal{G} = (\mathcal{A'}, \mpgen_0, \mpgen_1)$ from $\mathcal{G}_0$ simply by adding to the arena $\mathcal{A}$ a weight function $w_1$ that assigns a weight $0$ to each edge.
Stated formally, $\mathcal{A'} = (V,E, \zug{V_0,V_1}, w, w_1)$ such that for all $e \in E$, we have that $w_1(e)=0$.

Now consider that from a vertex $v \in V$, Player~0 has a winning strategy $\sigma_0$ in $\mathcal{G}_0$ such that $\mpgen(\sigma_0, \sigma_1) > c$ for all Player~1 strategies $\sigma_1$, and where $c$ is a rational.
We show that by playing $\sigma_0$ from $v$ in $\mathcal{G}$, we have that $\asv(v) > c$.
For every play $\pi \in \outv(\sigma_0)$ in $\mathcal{G}$, we have that $\mpgen_1(\pi)=0$, and Player~1 has a response $\sigma_\pi$ such that $\outv(\sigma_0, \sigma_{\pi})=\pi$.
In $\mathcal{G}$, Player~1 thus chooses a strategy that minimizes the mean-payoff of Player~0, and since in $\mathcal{G}_0$, we have that $\mpgen(\sigma_0, \sigma_1) > c$ for all strategies $\sigma_1$ of Player~1, it follows that $\asv(v) > c$.

Now in the other direction, consider that in $\mathcal{G}$, we have $\asv(v) > c$.
Thus from \cref{lem:witness_strategy}, there exists a strategy $\sigma_0$ for Player~0 such that $\asv(v)(\sigma_0) > c$.
Using similar arguments as above, we see that $\sigma_0$ is also a winning strategy in $\mathcal{G}_0$ giving a mean-payoff greater than $c$ to Player~0.

%% file: FiniteMemForASV.tex
\subparagraph{Finite memory strategies of Player~0}
In \cite{FGR20}, it has been shown that given a mean-payoff game $\mathcal{G}$, a vertex $v$ in $\mathcal{G}$, and a rational $c$, the problem of deciding if $\asvgen(v) > c$ is in {\sf NP}.
The use of an infinite memory strategy $\sigma_0$ for Player~$0$ such that $\asvgen(\sigma_0)(v) > c$ has been shown in \cite{FGR20}.
Here we give an improvement to that result in \cite{FGR20} showing that if 
$\asvgen(v) > c$, then there exists a finite memory strategy $\sigma_0$ of Player~0 such that $\asvgen(\sigma_0)(v) > c$.
The proof arguments are similar to that of \cref{ThmNpForASV}, and hence omitted.

Towards this, we first define the notion of a witness for $\asvgen$ as it appears in \cite{FGR20}.

\noindent\textbf{Witnesses for $\asvgen$} For a mean-payoff game $\mathcal{G}$, we associate with each vertex $v$ in $\mathcal{G}$, the following set of pairs of real numbers:
$\Lambda(v) = \{(c,d) \in \mathbb{R}^2 \mid v \vDash \ll 1 \gg \mpinf_0 \leqslant c \land \mpinf_1 \geqslant d \}$.
A vertex $v$ is said to be $(c,d)$-bad if $(c,d) \in \Lambda(v)$. Let $c' \in \mathbb{R}$. A play $\pi$ in $\mathcal{G}$ is called a $(c',d)$-witness of $\asvgen(v) > c$ if $(\mpinf_0(\pi), \mpinf_1(\pi)) = (c', d)$ where $c' > c$, and $\pi$ does not contain any $(c,d)$-bad vertex. A play $\pi$ is called a witness of $\asv(v) > c$ if it is a $(c',d)$-witness of $\asvgen(v) > c$ for some $c',d$. 

\input{Appendix/Threshold_Problem/ExistenceOfRegWitness}

Now using Lemma 8 in \cite{FGR20} (which is similar to \cref{LemPlaysAsWitnessForASV}, but in the context of $\asvgen$ instead of $\asv$), and using  \cref{ThmWitnessASVNoEpsFinMem}, we obtain the following theorem.

\input{Appendix/Threshold_Problem/ASVpseudopolynomial}

We now define the $\epsilon$-adversarial Stackelberg value and adversarial Stackelberg value when Player~0 is restricted to using finite memory strategies as :
\begin{equation*}
        \asv_\mathsf{FM}(v) =
        \quad \sup_{\sigma_0 \in \Sigma_0^{\fm}} \inf_{\sigma_1 \in \sf \br(\sigma_0)}  \mpinf_0(\outv(\sigma_0,\sigma_1))
\end{equation*}
\begin{equation*}
        \asvgen_\mathsf{FM}(v) =
        \quad \sup_{\sigma_0 \in \Sigma_0^{\fm}} \inf_{\sigma_1 \in \mathsf{BR}_1(\sigma_0)}  \mpinf_0(\outv(\sigma_0,\sigma_1))
\end{equation*}
where $\Sigma_0^{\fm}$ refers to the set of all finite memory strategies of Player~0.
We note that for every finite memory strategy $\sigma_0$ of Player~0, a best-response of Player~1 to $\sigma_0$ always exists as noted in \cite{FGR20}.
As a corollary of \cref{ThmNpForASV}, we observe that in a mean-payoff game $\mathcal{G}$, the $\asv$ from every vertex $v$ 
does not change even if Player~0 is restricted to using only finite memory strategies.
The result also holds for $\asvgen$ due to \cref{LemFinMemWitnessASVNonEps}.
Formally,
\begin{corollary}
\label{CorASVEqASVFinNonEps}
For all games $\mathcal{G}$, for all vertices $v$ in $\mathcal{G}$, and for all $\epsilon \!>\! 0$, we have that
$\asv_{\mathsf{FM}}(v) = \asv(v)$ and $\asvgen_{\mathsf{FM}}(v) = \asvgen(v)$.
\end{corollary}

Let $\asv(v) = c$ $(\asvgen(v) = c)$, for some $c \in \mathbb{Q}$.
Note that for every $c' \!<\! c$, there exists a finite memory strategy $\sigma_0^{FM}$ for Player~0 such that $\asv(\sigma_0^{FM})(v) > c'$ 
$(\asvgen(\sigma_0^{FM})(v) > c')$ leading to$\sup\limits_{\sigma_0 \in \Sigma_0^{\mathsf{FM}}} \asv(\sigma_0)(v)~\!=\!~\asv(v)=~c$ $(\sup\limits_{\sigma_0 \in \Sigma_0^{\mathsf{FM}}} \asvgen(\sigma_0)(v) = \asvgen(v) = c)$.
This gives $\asv_{\mathsf{FM}}(v) = \asv(v)$ ($\asvgen_{\mathsf{FM}}(v) = \asvgen(v)$).

This corollary is important from a practical point of view as it implies that both the $\asv$ and $\asvgen$ can be approached to any precision with a finite memory strategy. Nevertheless, we show in \cref{ThmAchiev} that infinite memory is necessary to achieve the exact $\asv$.

\subparagraph{Memoryless strategies of Player~0}
We now establish that the threshold problem is {\sf NP-complete} when Player~0 is restricted to play {\em memoryless} strategies.
First we define 
\begin{equation*}
        \asv_\mathsf{ML}(v) =
        \quad \sup_{\sigma_0 \in \Sigma_0^{\ml}} \inf_{\sigma_1 \in \sf \br(\sigma_0)}  \mpinf_0(\outv(\sigma_0,\sigma_1))
\end{equation*} and
\begin{equation*}
        \asvgen_\mathsf{ML}(v) =
        \quad \sup_{\sigma_0 \in \Sigma_0^{\ml}} \inf_{\sigma_1 \in \mathsf{BR}_1(\sigma_0)}  \mpinf_0(\outv(\sigma_0,\sigma_1))
\end{equation*}
where $\Sigma_0^{\ml}$ is the set of all memoryless strategies of Player~0.

\begin{theorem}
\label{ThmNPHardMemlessStrat}
For all mean-payoff games $\mathcal{G}$, for all vertices $v$ in $\mathcal{G}$, for all $\epsilon > 0$, and for all rationals $c$, the problem of deciding if $\asv_{\mathsf{ML}}(v) > c$ $(\asvgen_{\mathsf{ML}}(v) > c)$ is {\sf NP-Complete}.
\end{theorem}
\input{Appendix/Threshold_Problem/ThmNPHardMemlessStrat}

%% file: Appendix/Threshold_Problem/ExistenceOfRegWitness.tex
We state the following theorem which is similar to \cref{ThmWitnessASVFinMem}, but in the context of $\asvgen$ instead of $\asv$.
\begin{theorem}
\label{ThmWitnessASVNoEpsFinMem}
For all mean-payoff games $\mathcal{G}$, for all vertices $v$ in $\mathcal{G}$, and for all rationals $c$, we have that $\asvgen(v) > c$ if and only if there exists a regular $(c',d)$-witness of $\asvgen(v) > c$.
\end{theorem}
The proof of this theorem is exactly the same as that of 
\cref{ThmWitnessASVFinMem}, and hence omitted.

%% file: Appendix/Threshold_Problem/ASVpseudopolynomial.tex
\begin{theorem}
\label{LemFinMemWitnessASVNonEps}
For all mean-payoff games $\mathcal{G}$, for all vertices $v \in V$, and for all rationals $c$, if $\asvgen(v) > c$, then there exists a pseudopolynomial
memory strategy $\sigma_0$ for Player~0 such that $\asvgen(\sigma_0)(v) > c$.
\end{theorem}
The proof follows since as in the proof of \cref{ThmWitnessASVFinMem}, the values of $k$ and $\tau$ are polynomial in the size of $\mathcal{G}$, and the weights on the edges which are assumed to be given in binary.
The proof arguments are similar to that of \cref{ThmNpForASV}, and hence omitted.

%% file: Appendix/Threshold_Problem/ThmNPHardMemlessStrat.tex
\begin{proof}
The proof of hardness is a reduction from the partition problem while easiness is straightforwardly obtained by techniques used in the proof of Theorem~\ref{ThmNpForASV}.
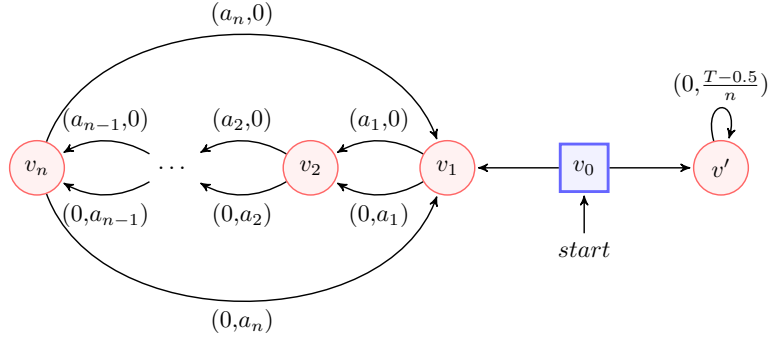
\begin{figure}[t]
    \begin{minipage}[b]{1\textwidth}    
    \centering
    \scalebox{0.9}{
    \begin{tikzpicture}[->,>=stealth',shorten >=1pt,auto,node distance=2cm,
                            semithick, squarednode/.style={rectangle, draw=blue!60, fill=blue!5, very thick, minimum size=7mm}]
          \tikzstyle{every state}=[fill=white,draw=black,text=black,minimum size=8mm]
        
          \node[squarednode]   (A)                    {$v_0$};
          \node[state, draw=red!60, fill=red!5]         (B_1) [left of=A]   {$v_1$};
          \node[state, draw=red!60, fill=red!5]         (B_2) [left of=B_1] {$v_2$};
          \node[state, draw = none, fill= none]         (B_)  [left of=B_2] {$\dotsc$};
          \node[state, draw=red!60, fill=red!5]         (B_n) [left of=B_]  {$v_n$};
          \node[state, draw=red!60, fill=red!5]         (C)   [right of=A]    {$v'$};
          \node[draw=none, fill=none, minimum size=0cm, node distance = 1.2cm]         (D) [below of=A] {$start$};
          \path (A)   edge                    node {}               (B_1)
                      edge                    node {}               (C)
                (B_1) edge [bend left, below] node {(0,$a_1$)}      (B_2)
                      edge [bend right, above] node {($a_1$,0)}     (B_2)
                (B_2) edge [bend left, below] node {(0,$a_2$)}      (B_)
                      edge [bend right, above] node {($a_2$,0)}     (B_)
                (B_)  edge [bend left, below] node {(0,$a_{n-1}$)}  (B_n)
                      edge [bend right, above] node {($a_{n-1}$,0)} (B_n)
                (B_n) edge [bend left=70, above] node {($a_n$,0)}      (B_1)
                      edge [bend right=70, below] node {(0,$a_n$)}     (B_1)
                (C)   edge [loop above]       node {(0,$\frac{T - 0.5}{n}$)} (C)
                (D) edge [left] node {} (A);
    
    \end{tikzpicture}}
    \caption{Reduction of the partition problem to the threshold problem where Player~0 is restricted to memoryless strategies.}
    \label{fig:partitionproblem}
    \end{minipage}
\end{figure}
Note that Player~0 can guess a memoryless strategy $\sigma_0$ from a vertex $v$ in {\sf NP}, and we consider the bi-weighted graph obtained by fixing the choices of Player~0 according to the strategy $\sigma_0$.
The set of maximal SCCs that are reachable from $v$ in this bi-weighted graph can be computed in linear time.
For each SCC $S$, we need to check that Player~1 cannot achieve $\mpinf_0 \leqslant c \land \mpinf_1 > d - \epsilon$ ($\mpinf_0 \leqslant c \land \mpinf_1 \geqslant d$) to ensure $\asv(\sigma_0) > c$ ($\asvgen(\sigma_0) > c$).
This can be done in polynomial time as has been done in the proof of \cref{ThmNpForASV}.

We prove the {\sf NP-}hardness result by reducing an {\sf NP-complete} problem, i.e., the \emph{partition problem}, to solving the threshold problem in a two-player nonzero-sum mean-payoff game. 
The partition problem is described as follows: Given a set of natural numbers $S = \{ a_1, a_2, a_3, \dotsc a_n \}$, to decide if we can partition $S$ into two sets $R$ and $M$ such that $\sum \limits_{a_i \in R} a_i = \sum \limits_{a_j \in M} a_j$.
W.l.o.g. assume that $\sum \limits_{a_i \in S} a_i = 2T$.

Given an instance of a partition problem, we construct a two-player nonzero-sum mean-payoff game as described in \cref{fig:partitionproblem} such that $\asv_{\mathsf{ML}}(v_0)> \frac{T-1}{n}$ for some $\epsilon < \frac{1}{2n}$ if and only if there exists a solution to the partition problem.


Consider the case, where a solution to the partition problem exists.
In the game in \cref{fig:partitionproblem}, we can construct a Player~0 strategy $\sigma_0$ as follows: from each Player~0 vertex $v_i$, for all $i \in \{1, 2, \dotsc n \}$, Player~0 chooses to play $(a_i, 0)$ (or the top-arrow) if $a_i \in R$ or $(0, a_i)$ (or the bottom-arrow) if $a_i \in M$ and from $v'$, she plays $v' \to v'$.
Thus, we have that from vertex $v_0$, if Player~1 chooses to play $v_0 \to v_1$, both players get a mean-payoff of $\frac{T}{n}$, but if he chooses to play $v_0 \to v'$, Player~1 gets a mean-payoff of $\frac{T-0.5}{n}$ which is not the $\epsilon$-best response for $\epsilon < \frac{1}{2n}$. Thus, we can see that $\asv_{\mathsf{ML}}(v_0) > \frac{T-1}{n}$.


Now consider the case, where the solution to the partition problem does not exist. 
We note that any memoryless strategy $\sigma_0$ for Player~0 would involve choosing to play from every vertex $v_i$, where $i \in \{1, 2, \dotsc n \}$, either $(a_i, 0)$ (or the top-arrow) or $(0, a_i)$ (or the bottom-arrow). This leads to a cycle from $v_1 \to v_n$, which either gives Player~1 a mean-payoff of $\leqslant \frac{T-1}{n}$ or Player~0 a mean-payoff of $\leqslant \frac{T-1}{n}$. In the first case, the $\epsilon$-best response of Player~1 would be to play $v_0 \to v'$, and thus $\asv_{\mathsf{ML}}(v_0) = 0$. For the second case, although Player~1's $\epsilon$-best response is to play $v_0 \to v_1$, Player~0 would only get a mean-payoff which is $\leqslant \frac{T-1}{n}$. Thus, $\asv_{\mathsf{ML}}(v_0)\leqslant \frac{T-1}{n}$.
Therefore in the two-player nonzero-sum mean-payoff game described in \cref{fig:partitionproblem},  we have that $\asv_{\mathsf{ML}}(v) > \frac{T-1}{n}$ if and only if there exists a solution to the partition problem for the set $S$.

We can construct a similar argument for the case of $\asvgen_{\mathsf{ML}}(v)$ by considering only best-responses (and not $\epsilon$-best responses) to the Player~0 strategy $\sigma_0$. Thus, in the two-player nonzero-sum mean-payoff game described in \cref{fig:partitionproblem},  we have that $\asvgen_{\mathsf{ML}}(v) > \frac{T-1}{n}$ if and only if there exists a solution to the partition problem for the set $S$.
\end{proof}

%% file: ComputeASV.tex
Here, we express the $\asv$ as a formula in the theory of reals by adapting a method provided in \cite{FGR20} for $\asvgen$. We then provide a new {\sf EXPTime} algorithm to compute the $\asv$ based on LP which in turn is applicable to $\asvgen$ as well.

\subparagraph{Extended mean-payoff game} Given a mean-payoff game $\mathcal{G} = (\mathcal{A},\langle \mpgen_0, \mpgen_1\rangle)$ with $\mathcal{A}=(V, E, \langle V_0, V_1 \rangle, w_0, w_1)$, we construct an extended mean-payoff game $\mathcal{G}^{\mathsf{ext}} = (\mathcal{A}^{\mathsf{ext}},\langle \mpgen_0, \mpgen_1\rangle)$, and whose vertices and edges are defined as follows.
The set of vertices is $V^{\mathsf{ext}} = V \times 2^V$. 
With every history $h$ in $\mathcal{G}$, we associate a vertex in $\mathcal{G}^{\mathsf{ext}}$ which is a pair $(v, P)$, where $v = \last(h)$ and $P$ is the set of the vertices traversed along $h$.
Accordingly the set of edges and the weight functions are respectively defined as $E^{\mathsf{ext}} = \{((v,P),(v', P')) \mid (v,v') \in E \land P' = P \cup \{v'\}\}$ and $w_i^{\mathsf{ext}}((v,P),(v', P')) = w_i(v, v')$, for $i \in \{0,1\}$.
We observe that there exists a bijection between the plays $\pi$ in $\mathcal{G}$ and the plays $\pi^{\mathsf{ext}}$ in $\mathcal{G}^{\mathsf{ext}}$ which start in vertices of the form $(v, \{v\})$, i.e. $\pi^{\mathsf{ext}}$ is mapped to the play $\pi$ in $\mathcal{G}$ that is obtained by erasing the second dimension of its vertices.
Note that the second component of the vertices of the play $\pi^{\mathsf{ext}}$ 
stabilises into a set of vertices of $\mathcal{G}$ which we denote by $V^{*}(\pi^{\mathsf{ext}})$. 

To relate the witnesses with the $\asv$ in the game $\mathcal{G}^{\mathsf{ext}}$, we introduce the following proposition. This proposition is an adaptation of Proposition 10 in \cite{FGR20} that is used to compute $\asvgen$.
\begin{proposition}
\label{PropExtGamePlayConverges}
For all mean-payoff games $\mathcal{G}$, the following holds:
\begin{itemize}
    \item Let $\pi^{\mathsf{ext}}$ be an infinite play in the extended mean-payoff game and $\pi$ be its projection on the original mean-payoff game $\mathcal{G}$ (over the first component of each vertex); the following properties hold: 
    \begin{itemize}
        \item For all $i < j$, if $\pi^{\mathsf{ext}}(i) = (v_i, P_i)$ and $\pi^{\mathsf{ext}}(j) = (v_j, P_j)$, then $P_i \subseteq P_j$
        \item $\underline{\mpgen}_i(\pi^{\mathsf{ext}}) = \underline{\mpgen}_i(\pi)$, for $i \in \{0,1\}$.
    \end{itemize}
    \item The unfolding of $\mathcal{G}$ from $v$ and the unfolding of $\mathcal{G}^{\mathsf{ext}}$ from $(v, \{v\})$ are isomorphic and so $\asv(v) = \asv(v, \{v\})$
\end{itemize}
\end{proposition}

By the first point of the above proposition and since the set of vertices of the mean-payoff game is finite, the second component of any play $\pi^{\mathsf{ext}}$, that keeps track of the set of vertices visited along $\pi^{\mathsf{ext}}$, stabilises into a set of vertices of $\mathcal{G}$ which we denote by $V^{*}(\pi^{\mathsf{ext}})$.

We now characterize $\asv(v)$ with the notion of witness introduced earlier and the decomposition of $\mathcal{G}^{\mathsf{ext}}$ into SCCs.
For a vertex $v$ in $V$, let $\mathsf{SCC}^{\mathsf{ext}}(v)$ be the set of strongly-connected components in $\mathcal{G}^{\mathsf{ext}}$ which are reachable from $(v, \{v\})$.
\begin{lemma}
\label{LemASVMaxSCCExtGamePlay}
 For all mean-payoff games $\mathcal{G}$ and for all vertices $v$ 
 in $\mathcal{G}$, 
 we have
\begin{equation*}
    \begin{split}
         \asv(v) = \max \limits_{S \in \mathsf{SCC}^{\mathsf{ext}}(v)} \sup \{ c \in \mathbb{R}  \mid & \exists \pi^{\mathsf{ext}}: \pi^{\mathsf{ext}} \textit{ is a witness for } \asv(v, \{v\}) > c \\ & \textit{ and } V^{*}(\pi^{\mathsf{ext}}) = S \}
    \end{split}
\end{equation*}
\end{lemma}
\input{Appendix/Computation/lem_ASV_is_max_SCC_GExt}

By definition of $\mathcal{G}^{\mathsf{ext}}$, for every $\mathsf{SCC}$ $S$ of $\mathcal{G}^{\mathsf{ext}}$, there exists a set $V^{*}(S)$ of vertices of $\mathcal{G}$ such that every vertex of $S$ is of the form $(v', V^*(S))$, where $v'$ is a vertex in $\mathcal{G}$.
Now, we define $\Lambda_S^{\mathsf{ext}} = \bigcup_{v \in V^{*}(S)} \Lambda^{\epsilon}(v)$ as the set of $(c,d)$ such that Player~1 can ensure $v \vDash \ll 1 \gg \mpinf_0 \leqslant c \land \mpinf_1 > d-\epsilon$ from some vertex $v \in S$.
The set $\Lambda_S^{\mathsf{ext}}$ can be represented by a formula $\Psi^{\epsilon}_{S}(x,y)$ in the first order theory of reals with addition, $\langle \mathbb{R}, +, < \rangle$, with two free variables.
Before we begin to prove the above, we refer to the following lemma which has been established in \cite{FGR20}.

\begin{lemma}
\label{LemPsiToThrOfReal}
 \emph{[Lemma 9 in \cite{FGR20}]}
 For all mean-payoff games $\mathcal{G}$, for all vertices $v$ in $\mathcal{G}$, and for all rationals $c, d$, we can effectively construct a formula $\Psi_v(x, y)$ of $\langle \mathbb{R}, +, < \rangle$ with two free variables such that $(c,d) \in \Lambda(v)$ if and only if the formula $\Psi_v(x, y)[x/c, y/d]$ is true.
\end{lemma}

Using the above lemma, we can now compute an effective representation of the infinite set of pairs $\Lambda^{\epsilon}(v)$ for each vertex $v$ of the mean-payoff game.
This is stated in the following lemma.

\begin{lemma}
\label{LemPsiEpsToThrOfReal}
 For all mean-payoff games $\mathcal{G}$, for all vertices $v$ in $\mathcal{G}$, for all $\epsilon > 0$, and for all rationals $c, d$, we can effectively  construct a formula $\Psi_v^{\epsilon}(x, y)$ of $\langle \mathbb{R}, +, < \rangle$ with two free variables such that $(c,d) \in \Lambda^{\epsilon}(v)$ if and only if the formula $\Psi_v^{\epsilon}(x, y)[x/c, y/d]$ is true.
\end{lemma}
\begin{proof}
From the definition of $\Lambda(v)$ from \cite{FGR20}, we know that a pair of real values $(c,d) \in \Lambda(v)$ if $v \vDash \ll 1 \gg \underline{\mpgen}_0 \leqslant c \land \underline{\mpgen}_1 \geqslant d$.
We now recall from the definition of $\Lambda^{\epsilon}(v)$ that $(c,d) \in \Lambda^{\epsilon}(v)$ if $v \vDash \ll 1 \gg \underline{\mpgen}_0 \leqslant c \land \underline{\mpgen}_1 > d-\epsilon$.
From this, we can see that
\begin{equation*}
    \Psi_v^{\epsilon}(x, y) \equiv \exists e > 0 \cdot \Psi_v(x, y - \epsilon + e)
\end{equation*}
\end{proof}

We can now state the following theorem about the computability of $\asv(v)$:

\begin{theorem}
\label{ThmComputeASV}
For all mean-payoff games $\mathcal{G}$, for all vertices $v$ in $\mathcal{G}$ and for all $\epsilon > 0$, the $\asv(v)$ can be effectively expressed by a formula in $\langle \mathbb{R}, +, < \rangle$, and can be computed from this formula.
\end{theorem}
\input{Appendix/Computation/thm_compute_ASV}
On the other hand, as illustrated in \cref{Ex:CompASVLP_new}, we note that if we fix a value of $c$ in every linear program corresponding to an $\mathsf{SCC}$ $S$, and replace the objective function $\mathsf{maximise}$ $c$ with the function $\mathsf{maximise}$ $\epsilon$, then we can find the supremum over $\epsilon$ which allows $\asv(v) > c$ in $S$.
Again, taking the maximum over all $\mathsf{SCC}$s reachable from $(v, \{v\})$, we get the largest $\epsilon$ possible so that we have $\asv(v) > c$.
Thus we get the following corollary.
\begin{corollary}
\label{thm:ComputeEpsilon}
For all mean-payoff games $\mathcal{G}$, for all vertices $v$ in $\mathcal{G}$, and for all $c \in Q$, we can compute in {\sf EXPTime} the maximum possible value of $\epsilon$ such that $\asv(v) > c$.
\end{corollary}

%% file: Appendix/Computation/lem_ASV_is_max_SCC_GExt.tex
\begin{proof}
First, we note the following sequence of inequalities:
\begin{equation*}
    \begin{split}
        \asv(v) &= \sup \{ c \in \mathbb{R} \mid \asv(v) \geqslant c \} \\
        &= \sup \{ c \in \mathbb{R} \mid \asv(v) > c \} \\
        &= \sup \{ c \in \mathbb{R} \mid \exists \pi: \pi \textit{ is a witness for } \asv(v) > c \} \\
        &= \sup \{ c \in \mathbb{R} \mid \exists \pi^{\mathsf{ext}}: \pi^{\mathsf{ext}} \textit{ is a witness for } \asv(v, \{v\}) > c \}\\
        &= \max \limits_{S \in \mathsf{SCC}^{\mathsf{ext}}(v)} \sup \{ c \in \mathbb{R} \mid \exists \pi^{\mathsf{ext}}: \pi^{\mathsf{ext}} \textit{ is a witness for } \asv(v, \{v\}) > c \textit{ and } V^{*}(\pi^{\mathsf{ext}}) = S \}
    \end{split}
\end{equation*}

The first two equalities follow from the definition of the supremum and that $\asv \in \mathbb{R}$. The third equality follows from \cref{ThmWitnessASVInfMem} that guarantees the existence of witnesses for strict inequalities. The fourth equality is due to the second point in \cref{PropExtGamePlayConverges}. The last equality is a consequence of first point in \cref{PropExtGamePlayConverges}.
\end{proof}

%% file: Appendix/Computation/thm_compute_ASV.tex




\begin{proof} 
\noindent To prove this theorem, we build a formula in $\langle \mathbb{R}, +, < \rangle$ that is true iff $\asv(v) = z$. 
Recall from \cref{LemASVMaxSCCExtGamePlay} 
that
 \begin{equation*}
    \asv(v) = \max \limits_{S \in \mathsf{SCC}^{\mathsf{ext}}(v)} \sup \{c \in \mathbb{R} \mid \exists \pi^{\mathsf{ext}}: \pi^{\mathsf{ext}} \textit{ is a witness for } \asv(v, \{v\}) > c \textit{ and } V^{*}(\pi^{\mathsf{ext}}) = S \}
 \end{equation*}
 
\noindent Since it is easy to express $\max \limits_{S \in \mathsf{SCC}^{\mathsf{ext}}(v)}$ in $\langle \mathbb{R}, +, < \rangle$, we concentrate on one $\mathsf{SCC}$ $S$ reachable from $(v, \{v\})$, and we show how to express
 \begin{equation*}
    \begin{split}
        \sup \{ c \in \mathbb{R} & \mid \exists \pi^{\mathsf{ext}}: \pi^{\mathsf{ext}} \textit{ is a witness for } \asv(v, \{v\}) > c \textit{ and } V^{*}(\pi^{\mathsf{ext}}) = S \}
     \end{split}
 \end{equation*}
in $\langle \mathbb{R}, +, < \rangle$.

\noindent Such a value of $c$ can be encoded by the following formula
\begin{equation*}
    \rho^S_{v}(c) \equiv \exists x,y \cdot x > c \land \Phi_S(x,y) \land \neg \Psi_S^{\epsilon}(c, y)
\end{equation*}
\noindent where $\Phi_S(x,y)$ is the symbolic encoding of $\Fmin{\CH{\CS}}$ in $\langle \mathbb{R}, +, < \rangle$ as defined in \cref{lemCHToPlay}. 
This states that the pair of values $(x,y)$ are the mean-payoff values realisable by some play in $S$. 
By \cref{LemPsiEpsToThrOfReal}, the formula $\neg \Psi_S^{\epsilon}(c, y)$ expresses that the play does not cross a $(c, y)^{\epsilon}$-bad vertex. So the conjunction $\exists x,y \cdot x > c \land \Phi_S(x,y) \land \neg \Psi_S^{\epsilon}(c, y)$ establishes the existence of a witness with mean-payoff values $(x,y)$ for the threshold $c$, and hence satisfying this formula implies that $\asv(v) > c$.
Now we consider the formula
\begin{equation*}
    \rho^S_{\max,v}(z) \equiv \forall e > 0 \cdot \rho^S_{v}(z-e) \land \forall c \cdot \rho^S_{v}(c) \implies c < z
\end{equation*}
which is satisfied by a value that is the supremum over the set of values $c$ such that $c$ satisfies the formula $\rho^S_{v}$, and hence the formula $\rho^S_{\max,v}(z)$ expresses
\begin{equation*}
    \begin{split}
        \sup \{ c \in \mathbb{R} & \mid \exists \pi^{\mathsf{ext}}: \pi^{\mathsf{ext}} \textit{ is a witness for } \asv(v, \{v\}) > c \textit{ and } V^{*}(\pi^{\mathsf{ext}}) = S \}
    \end{split}
\end{equation*}
From the formula $\rho^S_{\max,v}$, we can compute the $\asv(v)$ by quantifier elimination in
\begin{equation*}
    \max \limits_{S \in \mathsf{SCC}^{\mathsf{ext}}(v)} \exists z \cdot \rho^S_{\max,v}(z)
\end{equation*}
\noindent and obtain the unique value of $z$ that makes this formula true, and equals $\asv(v)$.
\end{proof}
\input{ExampleComputeASV}
\input{Appendix/Computation/thm_compute_ASV_EXP}

%% file: ExampleComputeASV.tex
\begin{example}
\label{Ex:CompASVTOR_new}
We illustrate the computation of $\asv$ with an example. 
Consider the mean-payoff game $\mathcal{G}$ depicted in \cref{fig:ach_example_new} and its extension $\mathcal{G}^{\mathsf{ext}}$ as shown in \cref{fig:extended_game_graph}.
\begin{figure}
    \centering
    \begin{minipage}[t]{.35\textwidth}
    \centering
        \scalebox{0.8}{
    \begin{tikzpicture}[->,>=stealth',shorten >=1pt,auto,node distance=2cm,
                        semithick, squarednode/.style={rectangle, draw=blue!60, fill=blue!5, very thick, minimum size=7mm}]
      \tikzstyle{every state}=[fill=white,draw=black,text=black,minimum size=8mm]
    
      \node[squarednode]   (A)                    {$v_0$};
      \node[state, draw=red!60, fill=red!5]         (B) [left of=A] {$v_1$};
      \node[state, draw=red!60, fill=red!5]         (C) [right of=A] {$v_2$};
      \node[draw=none, fill=none, minimum size=0cm, node distance = 1.2cm]         (D) [above of=A] {$start$};
      \path (A) edge [bend left, below] node {(1,1)} (B)
                edge              node {(0,1)} (C)
            (B) edge [loop above] node {(0,2)} (B)
                edge [bend left, above] node {(1,1)} (A)
            (C) edge [loop above] node {(0,1)} (C)
            (D) edge [left] node {} (A);
    \end{tikzpicture}}
    \caption{Example to calculate $\asv(v)$}
    \label{fig:ach_example_new}
    \end{minipage}
    \quad
    \begin{minipage}[t]{.55\textwidth}
    \centering
    \scalebox{0.8}{
    \begin{tikzpicture}[->,>=stealth',shorten >=1pt,auto,node distance=2cm,
                        semithick, squarednode/.style={rectangle, draw=blue!60, fill=blue!5, very thick, minimum size=7mm}]
      \tikzstyle{every state}=[text=black,minimum size=8mm]
      \node[squarednode]                    (A)                {$v_{0^2}'$};
      \node[state, draw=red!60, fill=red!5] (B)  [left of=A]   {$v_1'$};
      \node[squarednode]                    (A_) [below of=B]  {$v_{0^1}'$};
      \node[state, draw=red!60, fill=red!5] (C)  [right of=A]  {$v_{2^2}'$};
      \node[state, draw=red!60, fill=red!5] (C_) [right of=A_] {$v_{2^1}'$};
      \node[draw=none, fill=none, minimum size=0cm, node distance = 1.2cm]         (D) [left of=A_] {$start$};
      \path 
      (A)  edge [bend left, below] node {(1,1)} (B)
          edge                    node {(0,1)} (C)
      (A_) edge [, left]           node {(1,1)} (B)
          edge                    node {(0,1)} (C_)
      (B)  edge [loop above]       node {(0,2)} (B)
          edge [bend left, above] node {(1,1)} (A)
      (C)  edge [loop above]       node {(0,1)} (C)
      (C_) edge [loop below]       node {(0,1)} (C_)
      (D)  edge [left]             node {}      (A_);
    \end{tikzpicture}}
    \caption{Extended Mean-Payoff Game where $v_{0^1}' = (v_0, \{v_0\})$, $v_{0^2}' = (v_0, \{v_0, v_1\})$, $v_1' = (v_1, \{v_0, v_1\})$, $v_{2^1}' = (v_1, \{v_0, v_2\})$, and $v_{2^2}' = (v_2, \{v_0, v_1, v_2\})$}
    \label{fig:extended_game_graph}
    \end{minipage}
\end{figure}

Note that in $\mathcal{G}^{\mathsf{ext}}$ there exist three $\mathsf{SCC}$s which are $S_1 = \{v_{0^2}', v_1'\}, S_2 = \{v_{2^1}'\},$ and $S_3 = \{v_{2^2}'\}$. The $\mathsf{SCC}$s $S_2$ and $S_3$ are similar, and thus $\rho^{S_2}_{\max,v_0}(z)$ and $\rho^{S_3}_{\max,v_0}(z)$ would be equivalent.
We start with $\mathsf{SCC}$ $S_1$ that contains two cycles $v_1' \to v_1'$ and $v_{0^2}' \to v_1', v_1' \to v_{0^2}'$, and $\mathsf{SCC}$ $S_2$ contains one cycle $v_{2^1}' \to v_{2^1}'$.
Since $S_3$ is similar to $S_2$, we consider only $S_2$ in our example.
Thus, the set $\Fmin{\CH{\mathbb{C}(S_1)}}$ is represented by the Cartesian points within the triangle represented by $(0,2), (1,1)$ and $(0,1)$ \footnote{Note that the coordinate $(0,1)$ is obtained as the pointwise minimum over the two coordinates separately.} and $\Fmin{\CH{\mathbb{C}(S_2)}} = \{(0,1)\}$.
Thus, we get that $\Phi_{S_1}(x,y) \equiv (x \geqslant 0 \land x \leqslant 1) \land (y \geqslant 1 \land y \leqslant 2) \land (x+y) \leqslant 2$ and $\Phi_{S_2}(x,y) \equiv x = 0 \land y = 1$.
Now, we calculate $\Lambda^{\epsilon}(v_{0^2}')$, $\Lambda^{\epsilon}(v_{2^1}')$ and $\Lambda^{\epsilon}(v_1')$ for some value of $\epsilon$ less tan $1$.
We note that the vertex $v_1'$ is not $(0,2+\epsilon -\delta)^{\epsilon}$-bad, for all $0 < \delta < 1$, as Player~0 can always choose the edge $(v_1', v_{0^2}')$ from $v_1'$, thus giving Player~1 a mean-payoff of $1$.
\begin{figure}[ht]
    \begin{minipage}[t]{.25\textwidth}
    \centering
    \scalebox{1}{
    \begin{tikzpicture}
        \tkzInit[xmax=1,ymax=2]
        \tkzDefPoints{0/0/O,0/2/I,1/1/J,0/1/K}
        \tkzDrawXY[noticks,>=latex]
        \tkzLabelPoint[left](I){\scriptsize $P_1 (0,2)$}
        \tkzLabelPoint[right](J){\scriptsize $P_2 (1,1)$}
        \tkzLabelPoint[left](K){\scriptsize $P_3 (0,1)$}
        \tkzDrawPoints[shape=cross](I,J,K)
        \tkzFillPolygon[color=red!40, opacity=0.2](I,J,K)
        \tkzDrawPolySeg[color=red](I,J,K)
    \end{tikzpicture}}
    \caption{The red triangle represents the set of points in $\Phi_{S_1}$.}
    \label{fig:phi_points}
    \end{minipage}
    \quad
    \begin{minipage}[t]{.23\textwidth}
    \centering
    \scalebox{1}{\begin{tikzpicture}
        \tkzInit[xmax=1,ymax=2]
        \tkzDefPoints{0/0/O,0/1.3/I,1.3/1.3/J,1.3/0/K}
        \tkzDrawXY[noticks,>=latex]
        \tkzLabelPoint[left](I){\scriptsize $(0, 1+\epsilon)$}
        \tkzDrawPoints[shape=cross](I)
        \tkzFillPolygon[color=blue!40, opacity=0.2](O,I,J,K)
        \tkzDrawSegments[color = gray,->,
style=dashed](I,J)
    \end{tikzpicture}}
    \caption{The blue region under and excluding the line $y = (1-\epsilon)$ represents the set of points in $\Psi^{\epsilon}_{S_1}$ and $\Psi^{\epsilon}_{S_2}$.}
    \label{fig:psi_points}
    \end{minipage}
    \quad
    \begin{minipage}[t]{.45\textwidth}
    \centering
    \scalebox{1}{\begin{tikzpicture}
        \tkzInit[xmax=1,ymax=2]
        \tkzDefPoints{0/0/O,0/2/I,1/1/J,0/1/K}
        \tkzDefPoints{0/1.3/A,1.3/1.3/B,1.3/0/C, 0.7/1.3/D}
        \tkzDrawXY[noticks,>=latex]
        \tkzLabelPoint[left](I){\scriptsize $P_1 (0,2)$}
        \tkzLabelPoint[right](J){\scriptsize $P_2 (1,1)$}
        \tkzLabelPoint[left](K){\scriptsize $P_3 (0,1)$}
        \tkzLabelPoint[above right](D){\scriptsize $A (1-\epsilon,1+\epsilon)$}
        \tkzDrawPoints[shape=cross](I,J,K,D)
        \tkzFillPolygon[color=red!40, opacity=0.2](I,J,K)
        \tkzDrawPolySeg[color=red](I,J,K)
        \tkzLabelPoint[left](A){\scriptsize $(0, 1+\epsilon)$}
        \tkzDrawPoints[shape=cross](A)
        \tkzFillPolygon[color=blue!40, opacity=0.2](O,A,B,C)
        \tkzDrawSegments[color = gray,->,
style=dashed](A,B)
    \end{tikzpicture}}
    \vspace{-.1cm}
    \caption{The formula $\rho^{S_1}(c)$ is represented by the points in $\Phi_{S_1}$ and not in $\Psi^{\epsilon}_{S_1}$, i.e., the points in the triangle which are not strictly below the line $y = (1-\epsilon)$. Here, the max $c$ value is represented by point $A$.}
    \label{fig:rho_formula}
    \end{minipage}
\end{figure}
Additionally, the vertex $v_{0^2}'$ is both $(0,1+\epsilon-\delta)^{\epsilon}$-bad, for all $\delta > 0$, since Player~1 can choose the edge $(v_{0^2}', v_{2^2}')$ from $v_{0^2}'$, and $(1,1+\epsilon-\delta)^{\epsilon}$-bad, for all $\delta > 0$, since Player~1 can choose the edge $(v_{0^2}', v_{1}')$ from $v_{0^2}'$.
Thus, we get that $\Lambda^{\epsilon}(v_1') = \Lambda^{\epsilon}(v_{0^2}') = \{(c,d) \mid (c \geqslant 1 \land d < 1 + \epsilon)\} \bigcup \{(c,d) \mid (c \geqslant 0 \land d < 1 + \epsilon)\}$ which is the same as $\{(c,d) \mid (c \geqslant 0 \land d < 1 + \epsilon)\}$, and $\Lambda^{\epsilon}(v_{2^1}') = \{(c,d) \mid (c \geqslant 0 \land d < 1 + \epsilon)\}$.
Therefore, we have that $\Lambda_{S_1}^{\mathsf{ext}} = \Lambda_{S_2}^{\mathsf{ext}} = \{(c,d) \mid  (c \geqslant 0 \land d < 1 + \epsilon) \}$.
Hence, we get that $\Psi_{S_1}^{\epsilon}(x,y)= \Psi_{S_2}^{\epsilon}(x, y) \equiv (x \geqslant 0 \land y < 1 + \epsilon)$ as shown in \cref{fig:psi_points}.
From \cref{fig:rho_formula}, the formula $\rho^{S_1}(c)$ holds true for values of $c$ less than $(1 - \epsilon)$ and the formula $\rho^{S_2}(c)$ holds for values of $c$ less than 0. 
Hence, by assigning $(1 - \epsilon)$ to $x$, and $(1 + \epsilon)$ to $y$, we get that $\rho^{S_1}_{\max,v_0}(z)$ holds true for $z = (1-\epsilon)$.
Additionally, by assigning $0$ to $x$, and $1$ to $y$, we get that $\rho^{S_2}_{\max,v_0}(z)$ holds true for $z = 0$.
It follows that $\asv(v_0) = 1 - \epsilon$ for $\epsilon < 1$ as it is the maximum of the values over all the $\mathsf{SCC}$s. \qed
\end{example}

\subparagraph{An $\mathsf{EXPTime}$ algorithm for computing $\asv$}
Now we provide a new linear programming
based method that extends the previous approach to compute the $\asv$, and show that the value can be computed in $\mathsf{EXPTime}$. 
We first illustrate our approach with the help of the following
example by constructing the LP formulation for $\rho^{S}_v(c)$ for each {\sf SCC} $S$ thereby building the system of LPs for computing $\asv(v_0)$.

\begin{example}
\label{Ex:CompASVLP_new}
\input{ExampleComputeASV_EXP-TimeAlgorithm}
\qed
\end{example}


%% file: ExampleComputeASV_EXP-TimeAlgorithm.tex
\textnormal{We previously showed that the $\asv(v_0)$ can be computed by quantifier elimination of a formula in the theory of reals with addition. Now, we compute the $\asv(v_0)$ by solving a set of linear programs for every $\mathsf{SCC}$ in $\mathcal{G}^{\mathsf{ext}}$.
We recall that there are three $\mathsf{SCC}$s $S_1, S_2$ and $S_3$ in $\mathcal{G}^{\mathsf{ext}}$.}
\textnormal{From \cref{lemCHToPlay}, we have that $\Fmin{\CH{\mathbb{C}(S_i)}}$ for $i \in \{1, 2, 3\}$ can be defined using a set of linear inequalities.
Now recall that $\Fmin{\CH{\mathbb{C}(S_2}} = \Fmin{\CH{\mathbb{C}(S_3)}} = \{(0,1)\}$, and $\Fmin{\CH{\mathbb{C}(S_1)}}$ is represented by the set of points enclosed by the triangle formed by connecting the points $(0,1), (1,1)$ and $(0,2)$ as shown in \cref{fig:phi_points}, and
$\Lambda^{\epsilon}(v_{0^2}') = \Lambda^{\epsilon}(v_{2^1}') = \Lambda^{\epsilon}(v_{2^2}') = 
\Lambda^{\epsilon}(v_{1}') = \{(c,y) \:|\: c \geqslant 0 \land y < 1 + \epsilon \}$.}
\textnormal{Now, we consider the $\mathsf{SCC}$ $S_1$, and the formula $\neg \Psi_{S_1}^{\epsilon}$. We start this by finding the complement of $\Lambda^{\epsilon}(v_{0^2}')$ and $\Lambda^{\epsilon}(v_{1}')$, that is, $\overline{\Lambda}^{\epsilon}(v_{0^2}') = \overline{\Lambda}^{\epsilon}(v_1') = \mathbb{R} \times \mathbb{R} - \Lambda^{\epsilon}(v_{0^2}') = \mathbb{R} \times \mathbb{R} - \Lambda^{\epsilon}(v_1') = \{(c,y) \:|\: c < 0 \lor y \geqslant 1 + \epsilon \}$. Now, we get that $\neg \Psi_{S_1}^{\epsilon} = \overline{\Lambda}^{\epsilon}(v_{0^2}') \bigcap \overline{\Lambda}^{\epsilon}(v_{1}') = \{(c,y) \:|\: c < 0 \lor y \geqslant 1 + \epsilon \}$.}

\textnormal{Similarly for the $\mathsf{SCC}$ $S_2$ and $\mathsf{SCC}$ $S_3$, we calculate the complement of $\Lambda^{\epsilon}(v_{2^1}')$ and $\Lambda^{\epsilon}(v_{2^2}')$, that is, $\overline{\Lambda}^{\epsilon}(v_{2^2}') = \overline{\Lambda}^{\epsilon}(v_{2^1}') = \mathbb{R} \times \mathbb{R} - \Lambda^{\epsilon}(v_{2^1}') = \mathbb{R} \times \mathbb{R} - \Lambda^{\epsilon}(v_{2^2}') = \{(c,y) \:|\: c < 0 \lor y \geqslant 1 + \epsilon \}$ and obtain $\neg \Psi_{S_2}^{\epsilon} = \overline{\Lambda}^{\epsilon}(v_{2^1}') = \{(c,y) \:|\: c < 0 \lor y \geqslant 1 + \epsilon \}$ and $\neg \Psi_{S_3}^{\epsilon} = \overline{\Lambda}^{\epsilon}(v_{2^2}') = \{(c,y) \:|\: c < 0 \lor y \geqslant 1 + \epsilon \}$.
Note that the formulaes $\Phi_{S_2}(x,y)$ and $\Phi_{S_3}(x,y)$ are represented by the set of linear inequations $x=0 \wedge y=1$ and the formula $\Phi_{S_1}(x,y)$ is represented by the set of linear inequations $y \geqslant 1 \wedge y \leqslant 2 \wedge x \leqslant 1 \wedge (x+y) \leqslant 2$.
Now the formula $\rho_{v_0}^{S_1}(c)$ can be expressed using a set of linear equations and inequalities as follows: $x > c \land y \geqslant 1 \wedge y \leqslant 2 \wedge x \leqslant 1 \wedge (x+y) \leqslant 2 \land (c < 0 \lor y \geqslant 1+\epsilon)$ and the formula $\rho_{v_0}^{S_2}(c)$ can be expressed using a set of linear equations and inequalities as follows: $x > c \land x=0 \land y=1 \land (c < 0 \lor y \geqslant 1+\epsilon)$.
We maximise the value of $c$ in the formula $\rho_{v_0}^{S_1}(c)$ to get the following two linear programs:}
\textit{maximise $c$ in $(x > c \wedge y \geqslant 1 \wedge y \leqslant 2 \wedge x \leqslant 1 \wedge (x+y) \leqslant 2 \land c < 0)$} \textnormal{which gives a solution $\{0\}$ and}
\textit{maximise $c$ in $(x > c \wedge y \geqslant 1 \wedge y \leqslant 2 \wedge x \leqslant 1 \wedge (x+y) \leqslant 2 \land y \geqslant (1 + \epsilon))$}  \textnormal{which gives us a solution $ \{(1-\epsilon)\}$.
Similarly, maximising $c$ in the formulaes $\rho_{v_0}^{S_2}(c)$ and $\rho_{v_0}^{S_3}(c)$ would give us the following two linear programs:}
\textit{maximise c in $(x > c \land x = 0 \land y = 1 \land c < 0)$}  \textnormal{which gives a solution $\{0\}$ and}
\textit{maximise c in $(x > c \land x = 0 \land y = 1 \land y \geqslant (1 + \epsilon))$} \textnormal{which gives us a solution $\{0\}$.
Thus, we conclude that $\asv(v_0) = 1-\epsilon$ which is the maximum value amongst all the $\mathsf{SCC}$s.
Note that in an LP, the strict inequalities are replaced with non-strict inequalities, and computing the supremum in the objective function is replaced by maximizing the objective function.}

Again, for every $\mathsf{SCC}$ $S$ and for every LP corresponding to $S$, we can fix a value of $c$ and change the objective function to \textit{maximise $\epsilon$} from \textit{maximise $c$} in order to obtain the maximum value of $\epsilon$ that allows $\asv(v_0) > c$.
For example, consider the LP \textit{$(x > c \wedge y \geqslant 1 \wedge y \leqslant 2 \wedge x \leqslant 1 \wedge (x+y) \leqslant 2 \land y \geqslant (1 + \epsilon))$} in $\mathsf{SCC}$ $S_1$ and fix a value of $c$, and then maximize the value of $\epsilon$.
Doing this over all linear programs in an $\mathsf{SCC}$, and over all $\mathsf{SCC}$s, reachable from $v_0$ for a fixed $c$ gives us the supremum value of $\epsilon$ such that we have $\asv(v_0) > c$.

%% file: Appendix/Computation/thm_compute_ASV_EXP.tex

While no complexity upper bound for the computation of the ASV was reported in \cite{FGR20}, we show here that our procedure executes in {\sf EXPTime}. Indeed, for each SCC $S$, we express the formulae in $\langle \mathbb{R}, +, < \rangle$ as a set of exponentially many linear programs and thus, solve them in {\sf EXPTime}.
Therefore, we state the following theorem.
\begin{theorem}
\label{ThmComputeASVExpTime}
For all mean-payoff games $\mathcal{G}$, for all vertices $v$ in $\mathcal{G}$ and for all $\epsilon > 0$, the $\asv(v)$ can be computed in {\sf EXPTime}.
\end{theorem}

To prove \cref{ThmComputeASVExpTime}, we begin by stating and proving the following technical lemma.
\begin{lemma}
\label{lem:expression-inequalities}
The set $\bigcap_{v \in V^*(S)} \overline{\Lambda}^{\epsilon}(v)$ can be expressed as a union of exponentially many systems of strict and non-strict inequalities
\end{lemma}
\input{Appendix/Computation/thm_compute_ASV_LP}

We are now ready to prove \cref{ThmComputeASVExpTime}.
\begin{proof}[Proof of \cref{ThmComputeASVExpTime}]
First note that using 
    \cref{lemCHToPlay},
    for each $\mathsf{SCC}$ $S$ in $\mathcal{G}^{\mathsf{ext}}$, the set satisfying the formula $\Phi_S(x,y)$, which is the symbolic encoding of $\Fmin{\CH{\CS}}$, can be expressed as a set of \emph{exponentially} many inequalities.
    %
    Now recall that the formula $\Psi_S^{\epsilon}$ corresponds to $\bigcup_{v \in V^*(S)} \Lambda^{\epsilon}(v)$, and hence $\neg \Psi_S^{\epsilon}$ in the formula $\rho^S_{v}(c)$ corresponds to the set $\bigcap_{v \in V^*(S)} \overline{\Lambda}^{\epsilon}(v)$.
    We show in \cref{lem:expression-inequalities} that this set can be represented as a union of exponentially many systems of strict and non-strict inequalities.

For each $\mathsf{SCC}$ $S$ in the mean-payoff game $\mathcal{G}^{\mathsf{ext}}$, we see that the value satisfying the formula $\rho^S_{\max,v}$ can be expressed as a set of linear programs in the following manner.
      For each linear program, we have two variables $x$ and $y$ that represent the mean-payoff values of Player~0 and Player~1 respectively, corresponding to the formula $\Phi_S(x,y)$, and equivalently for plays in $\Fmin{\CH{\CS}}$, and a third variable $c$ represents the $c$ in the formula $\rho^S_{v}(c)$.
    %
      The 
      formula $\rho^S_{v}(c)$ can be expressed as a disjunction of a set of linear equations and inequalities, i.e., $\bigvee_i \mathsf{Cst}_i$ where each $\mathsf{Cst}_i$ is a conjunction of linear inequalities. 
      We use the variables $(c, d)$ in the linear inequalities which represent the set $\bigcap_{v \in V^*(S)}\overline{\Lambda}^{\epsilon}(v)$, and the variables $(x, y)$ in the linear inequalities representing the set $\Fmin{\CH{\CS}}$ as mentioned above. We also include the inequation $x > c$.
    Further, 
    the variable $d$ should assume the value of $y$ that corresponds to the mean-payoff of Player~$1$.
    In the LP formulation, each strict inequation is replaced by a non-strict inequation, and supremum in the objective function is replaced with maximizing the objective.
    %
      The set of linear programs we solve is \textit{maximise $c$ under the linear constraints $\mathsf{Cst}_i$}
    for each $\mathsf{Cst}_i \in \bigvee_i \mathsf{Cst}_i$. 

We solve the above sets of LPs for each $\mathsf{SCC}$ $S$ present in $\mathcal{G}$ to get a set of values satisfying the formula $\rho^S_{\max,v}$.
We choose the maximum of this set of values
which is the $\asv(v)$. 
The algorithm runs in {\sf EXPTime} as there can be exponentially many $\mathsf{SCC}$s.
\end{proof}

%% file: Appendix/Computation/thm_compute_ASV_LP.tex
\begin{proof}
From Lemma 4 of \cite{BR15}, we have that $\Lambda^{\epsilon}(v)$ can be represented as a finite union of polyhedra.
Considering a $d$-dimensional space, the set of points that satisfy the same set of linear inequalities forms an equivalence class, also called \emph{cells} \cite{BR15}.
Let $V_{\mathcal{G}}$ denote the set of mean-payoff coordinates of simple cycles in $\mathcal{G}$, and we have that $|V_{\mathcal{G}}| = \mathcal{O}(W \cdot |V|)^{\poly(d)}$.
Let $B(V_{\mathcal{G}})$ denote the set of geometric centres where each geometric centre is a centre of at most $d+1$ points from $V_{\mathcal{G}}$.
Thus $|B(V_{\mathcal{G}})| = \mathcal{O}(|V_{\mathcal{G}}|^{d+1})$.
From Lemma 6 of \cite{BR15} that uses Carath\'{e}odory baricenter theorem in turn, we have that $\Lambda^{\epsilon}(v)$ can be represented as a union of all cells that contain a point from $B(V_{\mathcal{G}})$ which is in $\Lambda^{\epsilon}(v)$.
Each cell is a polyhedron that can be represented by $\mathcal{O}(V_\mathcal{G})$ extremal points, or equivalently, by Theorem 3 of \cite{BR15}, by $\mathcal{O}(V_\mathcal{G}) \cdot 2^d$ inequalities.
It follows that $\overline{\Lambda}^{\epsilon}(v)$ can also be represented as a union of $\mathcal{O}(|V_{\mathcal{G}}|^{d+1})$ polyhedra.
Hence $\underset{v \in V^*(S)}{\bigcap} \overline{\Lambda}^{\epsilon}(v)$ can also be represented as a union of $\mathcal{O}(|V_{\mathcal{G}}|^{d+1})$ polyhedra.
Thus we have exponentially many linear programs corresponding to $\underset{v \in V^*(S)}{\bigcap} \overline{\Lambda}^{\epsilon}(v)$, since the weights on the edges are given in binary.
In our case, we have $d=2$.
Further, from \cref{lemCHToPlay}, for bi-weighted arena, we have that $\Fmin{\CH{\CS}}$ can be represented by $\mathcal{O}(|V_{\mathcal{G}}|^{2})$ linear inequalities, and hence $\Fmin{\CH{\CS}} \cap \underset{v \in V^*(S)}{\bigcap} \overline{\Lambda}^{\epsilon}(v)$ can be represented by exponentially many LPs.
Further, these LPs can be constructed in exponential time. \end{proof}

%% file: intro_additional_prop.tex
In this section, we first show that the $\asv$ is \emph{achievable}, i.e., there exists a Player~0 strategy that achieves $\asv$.
Then we study the memory requirement in strategies of Player~0 for achieving the $\asv$, as well as the memory requirement by Player~1 for playing the the $\epsilon$-best-responses. 

%% file: Achievability_v2.tex
\subparagraph{\bf Achievability of the $\asv$}
We formally define achievability as follows.
Given $\epsilon > 0$, we say that $\asv(v) = c$ is achievable from a vertex $v$, if there exists a strategy $\sigma_0$ for Player~0 such that $\forall \sigma_1 \in \br(\sigma_0) : \mpinf_0(\outv(\sigma_0, \sigma_1)) \geqslant c$.
We note that this result is in contrast to the case for $\asvgen$ as shown in \cite{FGR20}.
\begin{theorem}
    \label{ThmAchiev}
    For all mean-payoff games $\mathcal{G}$, for all vertices $v$ in $\mathcal{G}$, and for all $\epsilon > 0$, we have that the $\asv(v)$ is achievable.
\end{theorem}

The rest of this section is devoted to proving \cref{ThmAchiev}. We start by defining the notion of a witness for $\asv(\sigma_0)(v)$ for a strategy $\sigma_0$ of Player~0.

\subparagraph{Witness for $\asv(\sigma_0)(v)$}
Given a mean-payoff game $\mathcal{G}$, a vertex $v$ in $\mathcal{G}$, and an $\epsilon > 0$, we say that a play $\pi$ is a witness for $\asv(\sigma_0)(v) > c$ for a strategy $\sigma_0$ of Player~0 if 
\begin{inparaenum}[(i)]
\item $\pi \in \outv(\sigma_0)$, and
\item $\pi$ is a witness for $\asv(v) > c$ when Player~0 uses strategy $\sigma_0$
\end{inparaenum}
where the strategy $\sigma_0$ is defined as follows:
\begin{enumerate}
    \item $\sigma_0$ follows $\pi$ if Player~1 does not deviate from $\pi$.
    \item If Player~1 deviates $\pi$, then for each vertex $v \in \pi$, we have that $\sigma_0$ consists of a memoryless strategy that establishes $v \nvDash \ll 1 \gg \mpinf_0 \leqslant c \land \mpinf_1 > d - \epsilon$, where $d = \mpinf_1(\pi)$.
    The existence of such a memoryless strategy of Player~0 has been established in \cref{sec:ThresholdProblem}.
\end{enumerate}



Assume that the $\asv(v)$ cannot be achieved by a finite memory strategy.
We show that for such cases, it can indeed be achieved by an infinite memory strategy.
 
Let $\asv(v) = c$.
For every $c' < c$, from \cref{ThmNpForASV}, there exists a finite memory strategy $\sigma_0$ such that $\asv(\sigma_0)(v) > c'$, and recall from \cref{ThmWitnessASVFinMem} that there exists a corresponding regular witness.
%
First we state the following proposition.
\begin{proposition}
    \label{PropConvergenceStrategies}
    There exists a sequence of increasing real numbers, $c_1 < c_2 < c_3 < \dotsc < c$, such that the sequence converges to $c$, and a set of finite memory strategies $\sigma_0^1, \sigma_0^2, \sigma_0^3, \dotsc$ of Player~0 such that for each $c_i$, we have $\asv (\sigma_0^i)(v) > c_i$, and there exists a play $\pi^i$ that is a witness for $ \asv(\sigma_0^i)(v) > c_i$, where $\pi^i= \pi_{1}(l^{\alpha \cdot k_i}_{1} \cdot \pi_{2} \cdot l^{\beta \cdot k_i}_{2} \cdot \pi_{3})^{\omega}$, and $\pi_1, \pi_2$ and $\pi_3$ are simple finite plays, and $l_1, l_2$ are simple cycles in the arena of the game $\mathcal{G}$.
\end{proposition}
\input{Appendix/Achievability/PropConvergentStrategies}
These witnesses or plays in the sequence are regular, and they differ from each other only in the value of $k_i$ that they use.

To show that $\lim \limits_{i \to \infty} \asv(\sigma_0^i)(v) = c$, we construct a play $\pi^*$ that starts from $v$, follows $\pi^1$ until the mean-payoff of Player~0 over the prefix becomes greater than $c_1$.
Then for $i \in \{2, 3, \dots\}$, starting from $\first(l_1)$, it follows $\pi^i$, excluding the initial simple finite play $\pi_1$, until the mean-payoff of the prefix of $\pi^i$ becomes greater than $c_i$.
Then the play $\pi^*$ follows the prefix of the play $\pi^{i+1}$, excluding the initial finite play $\pi_1$, and so on.
%
Clearly, we have that $\mpinf_1(\pi^*) = c$. We let $\mpinf_1(\pi^*) = d = \alpha \cdot \mpgen_1(l_1) + \beta \cdot \mpgen_1(l_2)$.

For the sequence of plays $(\pi^i)_{i \in \mathbb{N}^+}$ which are witnesses for $(\asv(\sigma_0^i)(v) > c_i)_{i \in \mathbb{N}^+}$ for the strategies $(\sigma_0^i)_{i \in \mathbb{N}^+}$, we let $\mpinf_1(\pi_i) = d_i$. 
We state the following proposition.

\begin{proposition}
\label{PropMonotonicDi}
The sequence $(d_i)_{i \in \mathbb{N}^+}$ is monotonic, and it converges to $d$ in the limit.
\end{proposition}
\input{Appendix/Achievability/PropMotonicity}

The above two propositions establish the existence of an infinite sequence of regular witnesses $\asv(\sigma_0^i)(v) > c_i$ for  a sequence of increasing numbers $c_1 < c_2 < \dotsc < c$, such that the mean-payoffs of the witnesses are monotonic and at the limit, the mean-payoffs of the witnesses converge to $c$ and $d$ for Player~0 and Player~1 respectively. These observations show the existence of a witness $\pi^*$ which gives Player~0 a mean-payoff value at least $c$ and Player~1 a mean-payoff value equal to $d$. 
Assuming that Player~0 has a corresponding strategy $\sigma_0$, we show that Player~1 does not have an $\epsilon$-best response to $\sigma_0$ that gives Player~0 a payoff less than $c$.
Now, we have the ingredients to prove \cref{ThmAchiev}.

\begin{proof}[Proof of \cref{ThmAchiev}]

We consider a sequence of increasing numbers $c_1 < c_2 < c_3 < \dotsc < c$ such that for every $i \in \mathbb{N}^+$, by \cref{ThmNpForASV}, we consider a finite memory strategy $\sigma_0^i$ of Player~0 that ensures $\asv(\sigma_0^i)(v) > c_i$.

If the $\asv$ is not achievable, then there exists a strategy of Player~1 to enforce some play $\pi'$ such that $\mpinf_0(\pi') = c' < c$ and $\mpinf_1(\pi') = d' > d - \epsilon$. 
Now, we use the monotonicity of the sequence ($d_i$)$_{i \in \mathbb{N}^+}$ established in \cref{PropMonotonicDi} to show a contradiction. Since the sequence ($d_i$)$_{i \in \mathbb{N}^+}$ is monotonic, there can be two cases:
\vspace{-.1cm}
\begin{enumerate}
    \item The sequence ($d_i$)$_{i \in \mathbb{N}^+}$ is monotonically non-decreasing.
    \item The sequence ($d_i$)$_{i \in \mathbb{N}^+}$ is monotonically decreasing.
\end{enumerate}

We start with the first case where the sequence ($d_i$)$_{i \in \mathbb{N}^+}$ is non-decreasing.
Assume for contradiction that $\asv(v)$ is not achievable, i.e. Player~1 deviates from $\pi^*$ to enforce the play $\pi'$ such that $\mpinf_0(\pi') = c' < c$ and $\mpinf_1(\pi') = d'$.

Since ($d_i$)$_{i \in \mathbb{N}^+}$ is non-decreasing, and thus $d \geqslant d_i$ for all $i \in \mathbb{N}^+$, and since Player~1 can let his payoff to be reduced by an amount that is less than $\epsilon$ in order to reduce the payoff of Player~0, for all $i \in \mathbb{N}^+$ we have that $d' > d_i - \epsilon$.
We know that the sequence ($c_i$)$_{i \in \mathbb{N}^+}$ is increasing. Thus, there exists a $j \in \mathbb{N}$ such that $c' < c_j$.
Note that if $\pi'=\pi^{r}$ for some index $r$, we consider some $j$ which is also greater than $r$.

Now, consider the strategy $\sigma_0^j$ of Player~0 which follows the play $\pi_j$. We know that $\mpinf_0(\pi_j) > c_j$ and $\mpinf_1(\pi_j) = d_j$. We also know from \cref{LemPlaysAsWitnessForASV} that the play $\pi_j$ does not cross a $(c_j, d_j)^{\epsilon}$-bad vertex. 
Since by the construction of $\pi^*$, and by \cref{PropConvergenceStrategies}, the set of vertices appearing in the play $\pi^*$ is the same as the set of vertices appearing in the play $\pi_j$, for every vertex $v$ in $\pi^*$, we have that Player $1$ does not have a strategy such that 
$v \vDash \ll 1 \gg \mpinf_0 \leqslant c_j \land \mpinf_1 > d_j-\epsilon$.
Since $c' < c_j$ and $d' > d_j-\epsilon$, it also follows that Player $1$ does not have a strategy such that $v \vDash \ll 1 \gg \mpinf_0 \leqslant c' \land \mpinf_1 \geqslant d'$.
Stated otherwise, from the determinacy of multi-player mean-payoff games, we have that Player $0$ has a strategy to ensure $v \not \vDash \ll 1 \gg \mpinf_0 \leqslant c' \land \mpinf_1 \geqslant d'$ for every vertex $v$ appearing in $\pi^*$.
In fact, Player $0$ can ensure $v \not \vDash \ll 1 \gg \mpinf_0 \leqslant c' \land \mpinf_1 \geqslant d'$ by choosing the strategy $\sigma_{j'}$ for some $j' \ge j$.
Since $\asv(v) = \sup\limits_{\sigma_0 \in \Sigma_0} \asv(\sigma_0)(v)$, 
and the sequence 
$(c_i)_{i \in \mathbb{N}^+}$ is increasing, and we have that $\asv(\sigma_i)(v) > c_i$ for all $i \in \mathbb{N}^+$, it follows that the existence of $\pi'$ is a contradiction.

Now, we consider the case where the sequence ($d_i$)$_{i \in \mathbb{N}^+}$ is monotonically decreasing.
Again, assume for contradiction that $\asv(v)$ is not achievable, i.e., Player~1 deviates from $\sigma_0^*$ to enforce the play $\pi'$ such that $\mpinf_0(\pi') = c' < c$ and $\mpinf_1(\pi') = d'$.
Since the sequence ($d_i$)$_{i \in \mathbb{N}^+}$ is monotonically decreasing, we know that there must exist a $j \in \mathbb{N}$ such that 
\begin{inparaenum}[(i)]
\item $d' > d_j - \epsilon$, and $\forall i \geqslant j$, we have that $d' > d_i - \epsilon$, and
\item $c_j > c'$, which follows since ($c_i$)$_{i \in \mathbb{N}^+}$ is a strictly increasing sequence.
\end{inparaenum}
Thus for every vertex $v$ in $\pi^*$, Player $1$ does not have a strategy such that there exists a play $\pi$ in $\outv(\sigma_0^j)$, and $v \vDash \ll 1 \gg \mpinf_0 \leqslant c_j \land \mpinf_1 > d_j-\epsilon$.

Finally, using the fact that $c' < c_j$ and $d' > d_j-\epsilon$, the contradiction follows exactly as above where the sequence ($d_i$)$_{i \in \mathbb{N}^+}$ is monotonically non-decreasing.\footnote{We note that in the definition of $\epsilon$-best response in \cref{sec:prelim}, if we allow 
$\mpinf_1(\outv(\sigma_0,\sigma_1)) \geqslant \mpinf_1(\outv(\sigma_0,\sigma'_1)) - \epsilon$, 
then the existence of a $j$ for the case when the sequence ($d_i$)$_{i \in \mathbb{N}^+}$ is monotonically decreasing, does not necessarily hold true.}
\end{proof}


%% file: Appendix/Achievability/PropConvergentStrategies.tex
\begin{proof}
    Consider the play $\pi^i = \pi_{1i}(l^{\alpha \cdot k_i}_{1i} \cdot \pi_{2i} \cdot l^{\beta \cdot k_i}_{2i} \cdot \pi_{3i})^{\omega}$ which is a witness for $ \asv(\sigma_0^i)(v) > c_i$ for the strategy $\sigma_0^i$.
    Let $\mpinf_0(\pi^i)= c'_i > c_i$. 

    We have that $\mpinf_0(\pi^i)$ increases proportionally with $i$ as $\alpha \cdot k_i$ and $\beta \cdot k_i$ increase with increasing $k_i$. 
    This follows because we disregard the cases where $\asv(v)$ is achievable with some finite memory strategy of Player~0, i.e., we only consider the case where $\asv(v)$ is not achievable by a finite memory strategy of Player~0.

    
    Note that there are finitely many possible simple plays and simple cycles.
    Thus w.l.o.g. we can assume that in the sequence $(\pi^i)_{i \in \mathbb{N}^+}$, the finite plays $\pi_{1i}, \pi_{2i}, \pi_{3i}$, and the simple cycles $l_{1i}, l_{2i}$ are the same for different values of $i$.
    Thus, $\mpinf_0(\pi_{1i}(l^{\alpha \cdot k_i}_{1i} \cdot \pi_{2i} \cdot l^{\beta \cdot k_i}_{2i} \cdot \pi_{3i})^{\omega}) = c'_i > c_i, \mpinf_0(\pi_{1i}(l^{\alpha \cdot k_{i+1}}_{1i} \cdot \pi_{2i} \cdot l^{\beta \cdot k_{i+1}}_{2i} \cdot \pi_{3i})^{\omega}) = c'_{i+1} > c'_i, \mpinf_0(\pi_{1i}(l^{\alpha \cdot k_{i+2}}_{1i} \cdot \pi_{2i} \cdot l^{\beta \cdot k_{i+2}}_{2i} \cdot \pi_{3i})^{\omega}) = c'_{i+2} > c'_{i+1}$, and so on, and 
%
%
    the only difference in the strategies $\sigma_0^i$ as $i$ changes is the value of 
    $k_i$,
    i.e, we increase the value of $\alpha \cdot k_i$ and $\beta \cdot k_i$ with increasing $k_i$ such that the effect of $\pi_2$ and $\pi_3$ on the mean-payoff is minimised.
    Thus, at the limit, as $i \to \infty$, the sequence $(c_i)_{i \in \mathbb{N}^+}$ converges to $\alpha \cdot \mpgen_0(l_1) + \beta \cdot \mpgen_0(l_2) = c$.
\end{proof}

%% file: Appendix/Achievability/PropMotonicity.tex
\begin{proof}
Recall that $\mpinf_0(\pi^i)$ increases monotonically with increasing $i$.
Since the effect of the finite simple plays $\pi_2$ and $\pi_3$ decreases with increasing $\alpha \cdot k_i$ and $\beta \cdot k_i$, the mean-payoff on the second dimension also changes monotonically.
If $\alpha \cdot \mpgen_1(l_1) + \beta \cdot \mpgen_1(l_2) \geqslant \frac{w_1(\pi_2) + w_1(\pi_3)}{|\pi_2| + |\pi_3|}$, then the sequence $(d_i)_{i \in \mathbb{N}^+}$ is monotonically non-decreasing.
Otherwise, the sequence is monotonically decreasing.

The fact that this sequence converges to $d$ in the limit can be seen from the construction of $\pi^*$ as described above.
\end{proof}

%% file: ASVEpsilonExamples.tex

\subparagraph{{\bf Memory requirements of the players' strategies}}
First we show that there exists a mean-payoff game $\mathcal{G}$ in which Player~0 needs an infinite memory strategy to achieve the $\asv$.


\input{inf_memory_example}

We also show that exist mean-payoff games in which a finite memory (but not memoryless) strategy for Player $0$ can achieve the $\asv$.
\input{Appendix/Memory_Requirements/fin_memory_example}


Further, we show that there exist games such that for a strategy $\sigma_0$ of Player $0$, and an $\epsilon > 0$, there does not exist any finite memory best-response of Player $1$ to the strategy $\sigma_0$.

\input{inf_memory_player_1_example}

%% file: inf_memory_example.tex
\begin{theorem}
\label{ThmExNeedInfMem}
There exist a mean-payoff game $\mathcal{G}$, a vertex $v$ in $\mathcal{G}$, and an $\epsilon > 0$ such that Player $0$ needs an infinite memory strategy to achieve the
$\asv(v)$.
\end{theorem}
\input{Appendix/Memory_Requirements/inf_memory_example}

%% file: Appendix/Memory_Requirements/inf_memory_example.tex
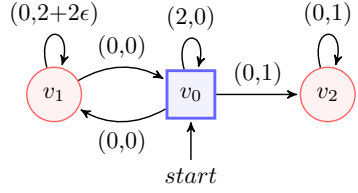
\begin{figure}[t]
\centering
\scalebox{0.9}{
    \begin{tikzpicture}[->,>=stealth',shorten >=1pt,auto,node distance=2cm,
                        semithick, squarednode/.style={rectangle, draw=blue!60, fill=blue!5, very thick, minimum size=7mm}]
      \tikzstyle{every state}=[fill=white,draw=black,text=black,minimum size=8mm]
    
      \node[squarednode]   (A)                    {$v_0$};
      \node[state, draw=red!60, fill=red!5]         (B) [left of=A] {$v_1$};
      \node[state, draw=red!60, fill=red!5]         (C) [right of=A] {$v_2$};
      \node[draw=none, fill=none, minimum size=0cm, node distance = 1.2cm]         (D) [below of=A] {$start$};
      \path (A) edge [bend left, below] node {(0,0)} (B)
                edge              node {(0,1)} (C)
                edge [loop above] node {(2,0)} (A)
            (B) edge [loop above] node {(0,2+2$\epsilon$)} (B)
                edge [bend left, above] node {(0,0)} (A)
            (C) edge [loop above] node {(0,1)} (C)
            (D) edge [left] node {} (A);
    \end{tikzpicture}}
    \caption{Finite memory strategy of Player~$0$ may not achieve $\asv(v_0)$. Also, no finite memory $\epsilon$-best response exists for Player~1 for the strategy $\sigma_0$ of Player~0.}
    \label{fig:no_finite_strategy}
\end{figure}
\begin{proof}
Consider the example in \cref{fig:no_finite_strategy}. We show that in this example the $\asv(v_0) = 1$, and that this value can only be achieved using an infinite memory strategy. Assume a strategy $\sigma_0$ for Player~0
such that the game is played in rounds.
In round $k$: 
\begin{inparaenum}[(i)]
    \item if Player~1 plays $v_0 \to v_0$ repeatedly at least $k$ times before playing $v_0 \to v_1$, then from $v_1$, play $v_1 \to v_1$ repeatedly $k$ times and then play $v_1 \to v_0$ and 
    move to round $k+1$;
    \item else, if Player~1 plays $v_0 \to v_0$ less than $k$ times before playing $v_0 \to v_1$, then from $v_1$ , play $v_1 \to v_0$.
\end{inparaenum}
Note that $\sigma_0$ is an infinite memory strategy.

The best-response for Player~1 to strategy $\sigma_0$ would be to choose $k$ sequentially as $k = 1, 2, 3, \dotsc$, to get a play $\pi = ((v_0)^i(v_1)^i)_{i \in \mathbb{N}}$. We have that $\mpinf_1(\pi) = 1+\epsilon$ and $\mpinf_0(\pi) = 1$. Player~1 can only sacrifice an amount that is less than $\epsilon$ to minimize the mean-payoff of Player~0, and thus he would not play $v_0 \to v_2$. 

We now show that $\asv(\sigma_0)(v_0) = \asv(v_0)$, and that no finite memory strategy of Player~0 can achieve an $\asv(v_0)$ of $1$.
First we observe in \cref{fig:no_finite_strategy} that a strategy $\sigma_1$ of Player~1 that prescribes playing the edge $v_0 \to v_2$ some time 
yields a mean-payoff of 1 for Player~1, and hence we conclude that $\sigma_1 \notin \br(\sigma_0)$. Player~1 cannot play any other strategy without increasing the mean-payoff of Player~0 and/or decreasing his own payoff.
We can see that Player~1 does not have a finite memory best-response strategy. Thus, the $\asv(\sigma_0)(v_0) = 1$.

We now prove the claim  $\asv(\sigma_0)(v_0) = \asv(v_0)$. For every strategy $\sigma_1$ of Player~1 such that $\sigma_1 \in \br(\sigma_0)$, we note that the higher the payoff Player~1 has, the lower is the payoff for Player~0. For every other strategy $\sigma_0'$ of Player~0, if best-response of Player~1 to $\sigma_0'$ gives a mean-payoff less than $1+\epsilon$, then Player~1 will switch to $v_2$, thus giving Player~0 a payoff of 0. If best-response of Player~1 to $\sigma_0'$ gives him a mean-payoff greater than $1+\epsilon$, then Player~0 will have a lower $\asv(\sigma_0')$.

Now we show that a finite memory strategy of Player~0 cannot achieve an $\asv(v_0)$ of $1$.
In order to achieve $\asv(v_0)$ of $1$, the edge $v_0 \to v_1$ and the edge $v_1 \to v_0$ need to be taken infinitely many times. 
If this does not happen, we note that either Player~1 would play $v_0 \to v_1$ and Player~0 would loop forever on $v_1$ or Player~1 would play $v_0 \to v_2$ and Player~0 would loop forever on $v_2$, both of which yield a mean-payoff of $0$ to Player~0.
Additionally, we note that the two edges $v_0 \to v_1$ and $v_1 \to v_0$ have edge weights $(0,0)$, i.e., it gives a payoff of $0$ to both players.
Therefore, we need a strategy which suppresses the effect of the edges on the mean-payoff of a play. 
This can be achieved by choosing these edges less and less frequently.
We now show that this cannot be done with any finite memory strategy, and we indeed need an infinite memory strategy as described by the Player~0 strategy $\sigma_0$.
Consider a finite memory strategy of Player~0 and an an infinite memory $\epsilon$-best-response for Player~1 for which no finite memory strategy exists. 
Given that he owns only one vertex which is $v_0$, such an infinite memory strategy can only lead to looping over $v_0$ more and more, thus giving him a payoff which is eventually 0.
Thus consider a finite memory response of Player~1 to the finite memory strategy of Player~0. 
Note that Player~0 would choose a finite memory strategy such that the best-response of Player~1 gives him a value of at least $1 + \epsilon$. Also since both players have finite memory strategies, the resultant outcome is a regular play over vertices $v_0$ and $v_1$. 
In every such regular play, the effect of the edge from $v_0$ to $v_1$ and the edge from $v_1$ to $v_0$ is non-negligible, and hence if the payoff of Player~1 is at least $1 + \epsilon$, the payoff of Player~0 will be less than 1. 
Thus no finite memory strategy can achieve an $\asv$ that is equal to 1.
\end{proof}

%% file: Appendix/Memory_Requirements/fin_memory_example.tex
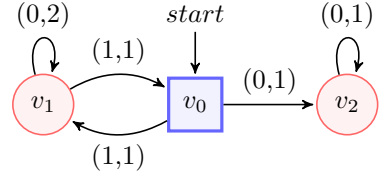
\begin{figure}[t]
    \begin{minipage}[b]{1\textwidth}
    \centering 
    \scalebox{1}{
    \begin{tikzpicture}[->,>=stealth',shorten >=1pt,auto,node distance=2cm,
                        semithick, squarednode/.style={rectangle, draw=blue!60, fill=blue!5, very thick, minimum size=7mm}]
      \tikzstyle{every state}=[fill=white,draw=black,text=black,minimum size=8mm]
    
      \node[squarednode]   (A)                    {$v_0$};
      \node[state, draw=red!60, fill=red!5]         (B) [left of=A] {$v_1$};
      \node[state, draw=red!60, fill=red!5]         (C) [right of=A] {$v_2$};
      \node[draw=none, fill=none, minimum size=0cm, node distance = 1.2cm]         (D) [above of=A] {$start$};
      \path (A) edge [bend left, below] node {(1,1)} (B)
                edge              node {(0,1)} (C)
            (B) edge [loop above] node {(0,2)} (B)
                edge [bend left, above] node {(1,1)} (A)
            (C) edge [loop above] node {(0,1)} (C)
            (D) edge [left] node {} (A);
    \end{tikzpicture}}
    \caption{an example in which a finite memory strategy for Player $0$ achieves the $\asv(v_0)$.}
    \label{fig:finite_strategy_response}
    \end{minipage}
\end{figure}
\begin{theorem}
\label[appendix]{ThmExNeedFinMem}
There exists a mean-payoff game $\mathcal{G}$, a vertex $v$ in $\mathcal{G}$, and an $\epsilon > 0$ such that a finite memory strategy of Player $0$ that can not be represented by any memoryless strategy suffices to achieve $\asv(v)$.
\end{theorem}
\begin{proof}
To show the existence of mean-payoff games in which Player $0$ can achieve the adversarial value with finite memory (but not memoryless) strategies, we consider the example in \cref{fig:finite_strategy_response}. We show that $\asv(v_0) = 1 - \epsilon$. Assume a strategy $\sigma_0$ for Player~0 defined as: repeat forever, from $v_1$ play $j$ times $v_1 \to v_1$, and then repeat playing $v_1 \to v_0$ for $k$ times, with $j$ and $k$ chosen such that mean-payoff for Player~0 is equal to $1 - \epsilon$.
For every rational $\epsilon$, such a $k$ always exists.
In this example, we have that $k=\frac{1-\epsilon}{2\epsilon}j$. The $\epsilon$-best-response of Player~1 to $\sigma_0$ is to always play $v_0 \to v_1$ as by playing this edge forever, Player~1 gets a mean-payoff equal to $1+\epsilon$, whereas if Player~1 plays $v_0 \to v_2$, then Player~1 receives a payoff of $1$. Since $1 \ngtr (1 + \epsilon) - \epsilon$, a strategy of Player~1 that chooses $v_0 \to v_2$ is not an $\epsilon$-best-response for Player~1, thus forcing Player~1 to play $v_0 \to v_1$. Thus $\asv(v_0)$ is achieved with a finite memory strategy of size $k$ for Player~0. Note that this size $k$ is a function of $\epsilon$.
\end{proof}

%% file: inf_memory_player_1_example.tex
\begin{theorem}
\label{ThmP1NeedInfMem}
There exist a mean-payoff game $\mathcal{G}$, an $\epsilon > 0$, and a Player~0 strategy $\sigma_0$ in $\mathcal{G}$ such that every Player~1 strategy $\sigma_1 \in \br(\sigma_0)$ is an infinite memory strategy.
\end{theorem}
\input{Appendix/Memory_Requirements/inf_memory_p1example}

%% file: Appendix/Memory_Requirements/inf_memory_p1example.tex
\begin{proof}
Consider the example in \cref{fig:no_finite_strategy}, and the strategy $\sigma_0$ of Player $0$ where the game is played in rounds as described in \cref{ThmExNeedInfMem}:
In round $k$: 
\begin{inparaenum}[(i)]
    \item if Player~1 plays $v_0 \to v_0$ repeatedly at least $k$ times before playing $v_0 \to v_1$, then from $v_1$, play $v_1 \to v_1$ repeatedly $k$ times and then play $v_1 \to v_0$ and 
    move to round $k+1$;
    \item else, if Player~1 plays $v_0 \to v_0$ less than $k$ times before playing $v_0 \to v_1$, then from $v_1$ , play $v_1 \to v_0$.
\end{inparaenum} 
We can see that for all 
finite memory strategies, Player~1 gets at most $1$, and thus no finite memory strategy is an $\epsilon$-best-response to $\sigma_0$. For each round $k$, Player~1's $\epsilon$-best-response is to play $v_0 \to v_0$ repeatedly at least $k$ times before playing $v_0 \to v_1$, which is an infinite memory strategy for Player~1.
\end{proof}

%% file: Appendix/Threshold_Problem/index.tex

\input{Appendix/Threshold_Problem/AdditionalDetailsThmWitnessASVFinMem}

%% file: Appendix/Threshold_Problem/AdditionalDetailsThmWitnessASVFinMem.tex
\section{Additional details in the proof of \cref{ThmWitnessASVFinMem}}
\label[appendix]{app:expressionForKAndTau}
\input{ExtensionOfProof-ExpressionsForTau}

%% file: ExtensionOfProof-ExpressionsForTau.tex
Below we compute the expressions for $k$ and $\tau$ for case 1 in the proof of 
\cref{ThmWitnessASVFinMem}.
We know that for play $\pi = \pi_1\rho_1\rho_2\rho_3\dots$ , where $\rho_i = l_1^{[\alpha i]}.\pi_2.l_2^{[\beta i]}.\pi_3$, constants $\alpha, \beta \in \mathbb{R}^{+}$ are chosen such that: 
\begin{align*}
    \alpha \cdot \mpinf_0(l_1) + \beta \cdot \mpinf_0(l_2) &= c'' \text{ where $c'' > c$}\\
    \alpha \cdot \mpinf_1(l_1) + \beta \cdot \mpinf_1(l_2) &= d \\
    \alpha + \beta &= 1 \\
\end{align*}
We assume here that $\mpinf_0(l_1) > \mpinf_0(l_2)$ and $\mpinf_1(l_1) < \mpinf_1(l_2)$. This implies that one simple cycle, $l_1$, increases Player~0's mean-payoff while the other simple cycle, $l_2$, increases Player~1's mean-payoff. We build a play $\pi' = \pi_1 \cdot (l_1^{[\alpha \cdot k]} \cdot \pi_2 \cdot l_2^{[\beta \cdot (k+\tau)]} \cdot \pi_3)^{\omega}$ where we choose constant $k \in \mathbb{N}$ and constant $\tau > 0$ such that $\mpinf_0(\pi') = c'$ and $\mpinf_1(\pi') = d$ for some $c' > c$. We try to express the conditions for $k$ and $\tau$ below:
\small
\begin{align*}
    &\mpinf_0(\pi') \\
    &= \frac{k\cdot\alpha\cdot w_0(l_1) + (k + \tau) \cdot \beta \cdot w_0(l_2) + w_0(\pi_2) + w_0(\pi_3)}{k\cdot\alpha\cdot|l_1| + (k+\tau)\cdot\beta\cdot|l_2| + |\pi_2| + |\pi_3|} \\
    &= \frac{k\cdot(\alpha\cdot w_0(l_1) + \beta \cdot w_0(l_2)) + \tau \cdot \beta \cdot w_0(l_2) + w_0(\pi_2) + w_0(\pi_3)}{k\cdot(\alpha\cdot|l_1| + \beta\cdot|l_2|) + \tau\cdot\beta + |\pi_2| + |\pi_3|}\\
    &\mpinf_1(\pi') \\
    &= \frac{k\cdot\alpha\cdot w_1(l_1) + (k + \tau) \cdot \beta \cdot w_1(l_2) + w_1(\pi_2) + w_1(\pi_3)}{k\cdot\alpha\cdot|l_1| + (k+\tau)\cdot\beta\cdot|l_2| + |\pi_2| + |\pi_3|}\\
    &= \frac{k\cdot(\alpha\cdot w_1(l_1) + \beta \cdot w_1(l_2)) + \tau \cdot \beta \cdot w_1(l_2) + w_1(\pi_2) + w_1(\pi_3)}{k\cdot(\alpha\cdot|l_1| + \beta\cdot|l_2|) + \tau\cdot\beta + |\pi_2| + |\pi_3|}
\end{align*}
\normalsize
Let $|\pi_2| + |\pi_3| = v$, $\alpha\cdot w_0(l_1) + \beta \cdot w_0(l_2) = x_0$, $\alpha\cdot w_1(l_1) + \beta \cdot w_1(l_2) = x_1$, $\alpha\cdot|l_1| + \beta\cdot|l_2| = y$, $w_0(\pi_2) + w_0(\pi_3) = z_0$ and $w_1(\pi_2) + w_1(\pi_3) = z_1$. We simplify the inequalities above to get:

\begin{align*}
    \mpinf_0(\pi') &= \frac{k\cdot x_0 + \tau \cdot \beta \cdot w_0(l_2) + z_0}{k\cdot y + \tau\cdot\beta + v}\\
    \mpinf_1(\pi') &= \frac{k\cdot x_1 + \tau \cdot \beta \cdot w_1(l_2) + z_1}{k\cdot y + \tau\cdot\beta + v}
\end{align*}

\noindent We know that $\mpinf_0(\pi') = c'$ and $\mpinf_1(\pi') = d$. Thus,

\begin{align*}
    \frac{k\cdot x_0 + \tau \cdot \beta \cdot w_0(l_2) + z_0}{k\cdot y + \tau\cdot\beta + v} &= c'\\
    \frac{k\cdot x_1 + \tau \cdot \beta \cdot w_1(l_2) + z_1}{k\cdot y + \tau\cdot\beta + v} &= d\\
    k\cdot x_0 + \tau \cdot \beta \cdot w_0(l_2) + z_0 &= c' \cdot (k\cdot y + \tau\cdot\beta + v)\\
    k\cdot x_1 + \tau \cdot \beta \cdot w_1(l_2) + z_1 &= d \cdot (k\cdot y + \tau\cdot\beta + v)\\
\end{align*}
Simplifying the above inequalities we get:
\begin{align*}
    k\cdot(x_0 - c'\cdot y) &= c'\cdot v + \tau \cdot \beta \cdot (c' - w_0(l_2)) - z_0 \\
    \tau \cdot \beta \cdot (w_1(l_2) - d) &= v\cdot d + k\cdot(d\cdot y - x_1) - z_1 \\
\end{align*}
Finally, after substitution of $\tau$ in the first inequality expression and further simplification of both expressions, we finally get:
\begin{align*}
   k &= \frac{(c'\cdot v - z_0)(w_1(l_2) - d) + (c' - w_0(l_2))(v\cdot d - z_1)}{(x_0 - c'\cdot y)(w_1(l_2) - d) - (d\cdot y - x_1)} \\
   \tau &= \frac{v\cdot d - z_1}{\beta \cdot(w_1(l_2) - d)} + \frac{d \cdot y - x_1}{w_1(l_2) - d}\cdot \frac{k}{\beta}
\end{align*}

\noindent The above two inequalities specify the range from which we can choose a suitable $k$ and $\tau$, such that the requirements $\mpinf_0(\pi') = c'$ and $\mpinf_1(\pi') = d$ are met. We note that $k$ and $\tau$ are polynomial in size of the game and the weights on the edges, i.e., $\mathcal{O} (k) = |W|^2 \cdot |V|^3$ and $\mathcal{O} (\tau) = |W|^3 \cdot |V|^5$.